\documentclass{article}
\usepackage{amssymb,amsfonts,amsmath,amsthm,amsopn,amstext,amscd,latexsym,xy,color,stmaryrd}
\usepackage{hyperref}
\theoremstyle{plain}
\input xy
\xyoption{all}
\newtheorem{introtheorem}{Theorem}
\newtheorem{theorem}{Theorem}[section]
\newtheorem{lemma}[theorem]{Lemma}
\newtheorem{proposition}[theorem]{Proposition}

\newtheorem{definition}[theorem]{Definition}

 \numberwithin{equation}{section}

\def\mf#1{\mathfrak{#1}}

\def\mb#1{\mathbb{#1}}
\def\tx#1{\textrm{#1}}

\def\Q{\mathbb{Q}}

\def\Z{\mathbb{Z}}

\def\rw{\rightarrow}

\def\lrw{\longrightarrow}

\def\sm{\smallsetminus}

\def\<{\langle}
\def\>{\rangle}

\renewcommand{\-}{\hyp{}}

\renewcommand{\)}{\textup{)}}

\DeclareMathOperator{\Hom}{Hom}

\DeclareMathOperator{\Ad}{Ad}
\DeclareMathOperator{\rank}{rank}

\DeclareMathOperator{\Gal}{Gal}

\DeclareMathOperator{\Aut}{Aut}
\DeclareMathOperator{\Stab}{Stab}

\DeclareMathOperator{\Lie}{Lie}

\DeclareMathOperator{\Pic}{Pic}

\DeclareMathOperator{\tr}{tr}

\DeclareMathOperator{\Spec}{Spec}

\newcommand{\cL}{{\mathcal L}}

\newcommand{\cO}{{\mathcal O}}

\newcommand{\cS}{{\mathcal S}}

\newcommand{\fre}{{\mathfrak e}}

\newcommand{\frg}{{\mathfrak g}}
\newcommand{\frh}{{\mathfrak h}}

\newcommand{\frl}{{\mathfrak l}}

\newcommand{\frp}{{\mathfrak p}}

\newcommand{\frs}{{\mathfrak s}}
\newcommand{\frt}{{\mathfrak t}}

\newcommand{\bbC}{{\mathbb C}}

\newcommand{\bbF}{{\mathbb F}}
\newcommand{\bbG}{{\mathbb G}}

\newcommand{\bbP}{{\mathbb P}}
\newcommand{\bbQ}{{\mathbb Q}}
\newcommand{\bbR}{{\mathbb R}}

\newcommand{\bbZ}{{\mathbb Z}}

\newcommand{\GL}{{\mathrm{GL}}}

\newcommand{\Sp}{{\mathrm{Sp}}}
\newcommand{\PSp}{{\mathrm{PSp}}}

\title{Arithmetic invariant theory and 2-descent for plane quartic curves}
\author{Jack A. Thorne\footnote{This research was partially conducted during the period the author served as a Clay Research Fellow.}}

\begin{document}
\maketitle
\begin{abstract}
Given a smooth plane quartic curve $C$ over a field $k$ of characteristic 0, with Jacobian variety $J$, and a marked rational point $P \in C(k)$, we construct a reductive group $G$ and a $G$-variety $X$, together with an injection $J(k)/2J(k) \hookrightarrow G(k) \backslash X(k)$. We do this using the Mumford theta group of the divisor $2 \Theta$ of $J$, and a construction of Lurie which passes from Heisenberg groups to Lie algebras.
\end{abstract}
\setcounter{tocdepth}{2}
\tableofcontents

\section*{Introduction}

\paragraph*{Motivation.} Let $C$ be a smooth, projective, geometrically connected algebraic curve over a field $k$ of characteristic 0, and let $J$ denote its Jacobian variety. It is of interest to calculate the group $J(k) / 2 J(k)$. For example, when $k = \bbQ$, this is often the first step in understanding the structure of the finitely generated abelian group $J(\bbQ)$. Calculating the group $J(k) / 2 J(k)$ is known as performing a 2-descent.

In order to calculate $J(k) / 2 J(k)$, it is often very useful to be able to understand this group in terms of explicit objects in representation theory. This is particularly the case if one wishes to understand the behaviour of the groups $J(k) / 2 J(k)$ as the curve $C$ is allowed to vary. A famous example is the description of this group in terms of binary quartic forms, in the case where $C = J$ is an elliptic curve \cite{Bir63}. More recently, Bhargava, Gross and Wang have given a similar description in the case where $C$ is an odd hyperelliptic curve, i.e.\ a hyperelliptic curve with a marked rational Weierstrass point $P \in C(k)$ \cite{Bha12, Wan13}. In this case, the group $J(k) / 2 J(k)$ is understood in terms of equivalence classes of self-adjoint linear operators with fixed characteristic polynomial.

The aim of this paper is to give an invariant-theoretic description of the group $J(k) / 2 J(k)$ when $C$ is a non-hyperelliptic genus 3 curve with a marked rational point $P \in C(k)$. Such a curve is canonically embedded as a quartic curve in $\bbP^2_k$, which explains the title of this paper. The set of such pairs $(C, P)$ breaks up into 4 natural families, according to the behaviour of the projective tangent line to $C$ at $P$ (these are described below).

Our results can be summarized in broad terms as follows: for each family of curves, we obtain a reductive group $G$ over $k$, an algebraic variety $X$ on which $G$ acts, and for each pair $x = (C, P)$ defined over $k$, a closed $G$-orbit $X_x \subset X$ and a canonical injection
\[ J(k) / 2 J(k) \hookrightarrow G(k) \backslash X_x(k). \]
If $k$ is separably closed, then the set $G(k) \backslash X_x(k)$ has a single element. In general, the set $X_x(k)$ of $k$-rational points breaks up into many $G(k)$-orbits, which become conjugate over the separable closure. The set of $G(k)$-orbits can be described in terms of Galois cohomology, and this allows us to make a link with the theory of 2-descent.

Two of the spaces $X$ that we construct are in fact linear representations, and our results in these cases (although not our proofs) parallel those in \cite[\S 4]{Bha12}. Bhargava and Gross apply the results of \emph{loc. cit.} to understand the average size of the 2-Selmer group of the Jacobian of an odd hyperelliptic curve over $\bbQ$. We hope that our results will have similar applications in the future,  but we do not pursue the study of Selmer groups in this paper.

The other two spaces we construct are global analogues of Vinberg's $\theta$-groups, which have been previously studied from the point of view of geometric invariant theory by Richardson \cite{Ric82}. We wonder if they can have similar applications in arithmetic invariant theory, and if there are similar and simpler spaces which are related, for example, to elliptic curves.

\paragraph*{Description of main results.} We now describe more precisely what we prove in this paper. Let $k$ be a field of characteristic 0. We are interested in the arithmetic of all pairs $(C, P)$ over $k$, where $C$ is a smooth non-hyperelliptic curve of genus 3, and $P \in C(k)$ is a marked rational point. We break up such pairs into 4 families, corresponding to the behaviour of the projective tangent line $\ell = T_P C$ in the canonical embedding:
\begin{enumerate} \item[Case $E_7$:] $\ell$ meets $C$ at exactly 3 points (the generic case).
\item[Case $\fre_7$:] $\ell$ meets $C$ at exactly 2 points, with contact of order 3 at $P$ ($\ell$ is a flex).
\item[Case $E_6$:] $\ell$ meets $C$ at exactly 2 points, with contact of order 2 at $P$ ($\ell$ is a bitangent line).
\item[Case $\fre_6$:] $\ell$ meets $C$ at exactly 1 point ($\ell$ is a hyperflex).
\end{enumerate}
The name for each case indicates the semisimple algebraic group or Lie algebra inside which we will construct the variety $X$ described above. The definitions are as follows:
\begin{enumerate} \item[Case $E_7$:] Let $H$ be a split adjoint simple group of type $E_7$, and let $\theta : H \to H$ be a split stable involution (see Proposition \ref{prop_stable_involution} below). We define $G$ to be the identity component of the $\theta$-fixed group $H^\theta$, and $Y$ to be the connected component of the identity in the $\theta$-inverted set $H^{\theta(h) = h^{-1}}$. (Equivalently, $Y$ can be realized as the quotient $H/G$.)
\item[Case $\fre_7$:] Let $H$, $\theta$, and $G$ be as in case $E_7$. We define $V$ to be the tangent space to $Y$ at the identity, where $Y$ is as in case $E_7$. Then $V$ is a linear representation of $G$, and can be identified with the $-1$-eigenspace of $\theta$ in $\frh = \Lie H$.
\item[Case $E_6$:] Let $H$ be instead a split adjoint simple group of type $E_6$, and let $\theta : H \to H$ be a split stable involution. We define $G$ to be the identity component of the $\theta$-fixed group $H^\theta$, and $Y$ to be the connected component of the identity in the $\theta$-inverted set $H^{\theta(h) = h^{-1}}$.
\item[Case $\fre_6$:] Let $H$, $\theta$, and $G$ be as in case $E_6$. We define $V$ to be the tangent space to $Y$ at the identity, where $Y$ is as in case $E_6$. Equivalently, $V = \frh^{\theta = -1} \subset \frh$.
\end{enumerate}
In case $E_7$ or $E_6$, we let $X=Y$. In case $\fre_7$ or $\fre_6$, we let $X=V$. In each case the open subscheme $X^\text{s} \subset X$ of geometric stable orbits (i.e.\ closed orbits with finite stabilizers) is non-empty, and can be realized as the complement of a discriminant hypersurface. A Chevalley restriction theorem holds, and if $k$ is separably closed then two elements $x, y \in X^\text{s}(k)$ are $G(k)$-conjugate if and only if they have the same image in the categorical quotient $X  \! \sslash \! G$. (We remark that the quotients $V \! \! \sslash \! G$ are abstractly isomorphic to affine space. This is not so for the quotients $Y \! \sslash \! G$, although it would be so if in their definition we replaced the adjoint group $H$ by its simply connected cover.) The spaces $V$ are linear representations of $G$ of the type arising from Vinberg theory, and have been studied in the context of arithmetic invariant theory in e.g.\ \cite{Tho13}. The spaces $Y$ are a `global' analogue of the representations $V$.

Our first main result is the construction of a point of $G(k) \backslash X^s(k)$ which corresponds to the trivial element of the group $J(k) / 2 J(k)$:
\begin{introtheorem}\label{thm_intro_base_orbit}[Theorem \ref{thm_base_orbits}] \begin{enumerate} \item In case $E_7$ or $E_6$, let $\cS$ denote the functor $k\text{-alg} \to \text{Sets}$ which classifies pairs $(C, P)$, where $C$ is a smooth, non-hyperelliptic curve of genus 3, and $P$ is a point of $C$ as above. Then there is a canonical map
\[ \cS(k) \to G(k) \backslash Y^\text{s}(k). \]
If $k$ is separably closed, then this map is bijective.
\item In case $\fre_7$ or $\fre_6$, let $\cS$ denote the functor $k\text{-alg} \to \text{Sets}$ which classifies tuples $(C, P, t)$, where $C$ is a smooth non-hyperelliptic curve of genus 3, $P$ is a point of $C$ as above, and $t$ is a non-zero element of the Zariski tangent space of $C$ at $P$. Then there is a canonical map
\[ \cS(k) \to G(k) \backslash V^\text{s}(k). \]
If $k$ is separably closed, this map is bijective.
\end{enumerate}
\end{introtheorem}
In any of the above cases, given $x \in \cS(k)$ corresponding to a tuple $(C, P, \dots)$, we write $J_x$ for the Jacobian of $C$ and $X_x \subset X$ for the geometric stable orbit containing the image of $x$, where again $X = Y$ in case $E_7$ or $E_6$, and $X = V$ in case $\fre_7$ or $\fre_6$. As noted above, $G(k)$ acts transitively on $X_x(k)$ if $k$ is separably closed, but in general this is not the case; instead, the orbits comprising $G(k) \backslash X_x(k)$ can be described in terms of Galois cohomology. Our main theorem shows how to construct orbits in $G(k) \backslash X_x(k)$ using rational points of $J_x(k)$:
\begin{introtheorem}\label{thm_intro_orbits}[Theorem \ref{thm_groups_and_descent}] Let notation be as above. Then there is a canonical injection $J_x(k)/2J_x(k) \hookrightarrow G(k) \backslash X_x(k)$. The image of the identity element of $J_x(k)$ is the image of $x$ under the map of Theorem \ref{thm_intro_base_orbit}.
\end{introtheorem}
We observe that the Jacobian $J_x$ depends only on the curve $C$, but the set $G(k) \backslash X_x(k)$ depends on the choice of auxiliary data; an analogous situation arises when doing 2-descent on the Jacobian of a hyperelliptic curve which has more than one $k$-rational Weierstrass point.
\paragraph*{Methods.}
The methods we adopt to prove Theorems 1 and 2 seem to be different to preceding work of a similar type. This reflects the fact that we are now in the territory of exceptional groups, whereas e.g.\ 2-descent on hyperelliptic curves can be understood using the invariant theory of Vinberg $\theta$-groups which are constructed inside classical groups (in fact, groups of type $A_n$).

Our starting point is a classical geometric construction. For concreteness, we describe what happens just in the case of type $E_6$. Let us therefore take a smooth, non-hyperelliptic curve $C$ over $\bbC$ of genus 3, and let $P \in C(\bbC)$ be a marked point where the projective tangent line in the canonical embedding is a bitangent line. The double cover $\pi : S \to \bbP^2$ branched  over $C$ is a del Pezzo surface of degree 2, and the strict transform of $\ell$ is union of two $(-1)$-curves; blowing down one of these, we obtain a smooth cubic surface $S$.

There is a well-known connection between cubic surfaces and the root system of type $E_6$: let $\Lambda = K_S^\perp \subset H^2(S, \bbZ)$ denote the orthogonal complement of the canonical class of $S$. Then $\Lambda$ is in fact a root lattice of type $E_6$. This does not immediately provide a relation with geometric invariant theory because there is no functorial construction of a reductive group from a root lattice.

However, Lurie \cite{Lur01} has observed that one can construct in a functorial way the  group $H$ corresponding to $\Lambda$ given the additional data of a \emph{double cover} of $V = \Lambda/2\Lambda$, i.e.\ a group extension
\begin{equation}\label{eqn_intro_lurie_data} \xymatrix@1{ 1 \ar[r] & \{ \pm 1 \} \ar[r] & \widetilde{V} \ar[r] & V \ar[r] & 1}
\end{equation}
satisfying some additional conditions; in particular, that the quadratic form $q : V \to \bbF_2$ corresponding to this extension agrees with the one derived from the natural quadratic form on $\Lambda$.

It turns out that the realization of the cubic surface $X$ using the plane quartic curve $C$ is exactly the data required for input into Lurie's construction. Indeed, let $J$ denote the Jacobian of the curve $C$. Then $J$ has a natural principal polarization $\Theta$, and associated to $\cL = 2\Theta$ is the Mumford theta group
\begin{equation}\label{eqn_intro_mumford_sequence} \xymatrix{ 1 \ar[r] & \{ \pm 1 \} \ar[r] & \widetilde{H}_\cL \ar[r] & J[2] \ar[r] & 1.}
\end{equation}
(More precisely, the Mumford theta group is a central extension of $J[2]$ by $\bbG_m$. The presence of the odd theta characteristic corresponding to the bitangent $\ell$ allows us to refine it to an extension by $\{ \pm 1 \}$.) We show that there is a canonical isomorphism $J[2] \cong \Lambda/2\Lambda$; pushing out the sequence (\ref{eqn_intro_mumford_sequence}) by this isomorphism, we obtain a sequence of type (\ref{eqn_intro_lurie_data}), to which Lurie's construction applies. We thus obtain from the data $(C, \ell)$ an algebraic group of type $E_6$. (We remark here that the isomorphism $J[2] \cong \Lambda / 2 \Lambda$ is well-known and classical; see for example \cite[IX, \S1]{Dol88}. We thank the anonymous referee for this reference.)

The principle underlying this paper is that the construction outlined above is sufficiently functorial that we can recover the arithmetic situation over any field $k$ of characteristic 0 simply by Galois descent. To construct the orbits whose existence is asserted by Theorem \ref{thm_intro_orbits}, we simply twist the extension (\ref{eqn_intro_mumford_sequence}). More precisely, we recall in \S \ref{sec_heisenberg_and_descent} below how a point of $J_x(k)$ gives rise to a twisted form of the Heisenberg group $\widetilde{H}_\cL$. We then construct additional orbits by applying our version of Lurie's construction to this twisted Heisenberg group.

\paragraph*{Other remarks.} There are some minor subtleties in our construction that we remark on now. One point is that in cases $\fre_6$, $\fre_7$, we associate orbits not to pairs $(C, P)$ but to triples $(C, P, t)$, where $t$ is a non-zero Zariski tangent vector at the point $P$. This reflects the fact that the space $X$ constructed in this case has an extra symmetry: it is a linear representation of the reductive group $G$, so we are free to multiply elements by scalars. This scaling corresponds to scaling the tangent vector $t$. A similar feature appears in the work of Bhargava--Gross \cite{Bha12}, where it allows one to `clear denominators' when working over $\bbQ$, and restrict to integral orbits.

Another point is that in the geometric construction sketched above, we associate a point to a pair $(C, \ell)$, and do not need the point $P$ which gives rise to the bitangent $\ell$. Of course, $\ell$ being fixed, there are exactly two possible choices of point $P$. It turns out that in each case, the data of the point $P$ is exactly the data required to rigidify the picture so that we obtain the expected bijection (as in Theorem 1) when $k$ is separably closed. This is an essential feature, since we rely heavily on Galois descent.

Our modified version of Lurie's construction associates to an appropriate extension $\widetilde{V}$ with action by the absolute Galois group of $k$ a triple $(\frh, \frt, \theta)$ consisting of a Lie algebra over $k$ of the correct Dynkin type, a Cartan subalgebra $\frt \subset \frh$, and a stable involution $\theta$ of $\frh$ which acts as multiplication by $-1$ on $\frt$. For arithmetic applications, we extend this construction in a surprising way: we show that a representation of the group $\widetilde{V}$ appearing in the extension (\ref{eqn_intro_lurie_data}), and on which $-1$ acts as multiplication by $-1$, gives rise to a representation of the $\theta$-fixed Lie algebra $\frh^{\theta}$.

The features of these constructions suggest that they should have an inverse, i.e.\ that given a tuple $(\frh, \frt, \theta)$ consisting of a simple Lie algebra $\frh$ over $k$, a Cartan subalgebra $\frt \subset \frh$ and an involution $\theta$ of $\frh$ which acts as $-1$ on $\frt$, one should be able to pass in the opposite direction to obtain a root lattice $\Lambda$ with $\Gamma_k$-action and an extension $\widetilde{V}$ of $V = \Lambda/2\Lambda$ of type (\ref{eqn_intro_lurie_data}). The existence of such an inverse has been shown by Tasho Kaletha, and appears in an appendix to this paper. He finds the group $\widetilde{V}$ inside the simply connected cover of the group $G = (H^\theta)^\circ$, where $H$ is the adjoint simple group over $k$ with Lie algebra $\frh$. In \S \ref{sec_reals}, we apply these results to calculate the number of orbits with given invariants in the case $k = \bbR$.
\paragraph*{Organization of this paper.}
In \S \ref{section_background} below, we recall some basic facts about quadratic forms, 2-descent for abelian varieties, and the invariant theory of the $G$-varieties under consideration here. In \S \ref{sec_group_with_involution} we describe our modifications to Lurie's constructions. In \S\ref{sec_plane_quartic_curves} we apply these constructions to the geometry of  plane quartics, in order to arrive at the results described in this introduction.  We conclude in \S \ref{sec_reals} with an explicit example in the case $k = \bbR$.
\paragraph*{Acknowledgements.}

I am grateful to Manjul Bhargava, Dick Gross, and Tasho Kaletha for many interesting conversations. I would like to thank Tasho again for writing the appendix to this paper. Finally, I thank the anonymous referee for their helpful comments.

\paragraph*{Notation.}

Throughout this paper, $k$ will denote a field of characteristic 0, and $k^s$ a fixed separable closure of $k$. We write $\Gamma_k = \Gal(k^s/k)$. If $X$ is a $k$-vector space or a scheme of finite type over $k$, then we write $X_{k^s}$ for the object obtained by extending scalars to $k^s$. If $X$ is a smooth projective variety over $k$, then we write $K_X$ for its canonical class. If $G, H, \dots$ are connected algebraic groups over $k$, then we use gothic letters $\frg, \frh, \dots$ to denote their Lie algebras. If $H$ is an algebraic group over $k$, then we write $H^1(k, H)$ for the continuous cohomology set $H^1(\Gamma_k, H(k^s))$, where $H(k^s)$ is endowed with the discrete topology. If $\theta$ is an involution of $H$, then we write $H^\theta$ for the closed subgroup of $H$ consisting of $\theta$-fixed elements, and $\frh^\theta$ for the Lie algebra of $H$ (equivalently, the $+1$-eigenspace of the differential of $\theta$ in $\frh$). We will make use of the equivalence between commutative finite $k$-groups and $\bbZ[\Gamma_k]$-modules of finite cardinality (given by $H \mapsto H(k^s)$).

By definition, a lattice $(\Lambda, \langle \cdot, \cdot \rangle)$ is a finite free $\bbZ$-module $\Lambda$ together with a symmetric and positive-definite bilinear form $\langle \cdot, \cdot \rangle : \Lambda \times \Lambda \to \bbZ$. We define $\Lambda^\vee = \{ \lambda \in \Lambda \otimes_\bbZ \bbQ \mid \langle \lambda, \Lambda \rangle \subset \bbZ \}$, which is naturally identified with $\Hom(\Lambda, \bbZ)$. We call $\Lambda$ a (simply laced) root lattice if it satisfies the following additional conditions:
\begin{itemize}
\item For each $\lambda \in \Lambda$, $\langle \lambda, \lambda \rangle$ is an even integer.
\item The set $\Gamma = \{ \lambda \in \Lambda \mid \langle \lambda, \lambda \rangle = 2 \}$ generates $\Lambda$ as an abelian group.
\end{itemize}
In this case, $\Gamma$ is a simply laced root system, each $\gamma \in \Gamma$ being associated with the simple reflection $s_\gamma(x) = x - \langle x, \gamma \rangle \gamma$. If $\Gamma$ is \emph{irreducible}, then it is a root system of type $A$, $D$, or $E$. In any case, we write $W(\Lambda) \subset \Aut(\Lambda)$ for the Weyl group of $\Gamma$, a finite group generated by the simple reflections $s_\gamma$, $\gamma \in \Gamma$.

In several places, we will consider central group extensions of the form
\[ \xymatrix@1{ 1 \ar[r] & \{ \pm 1 \} \ar[r] & \widetilde{E} \ar[r] & E \ar[r] & 1.} \]
If $\widetilde{e} \in \widetilde{E}$, then we will write $-\widetilde{e}$ for the element $(-1) \cdot \widetilde{e}$. We note that this is not necessarily equal to $\widetilde{e}^{-1}$. We write $e$ for the image of $\widetilde{e}$ in $E$.

\section{Background}\label{section_background}

We first recall some background material. For proofs of the results in \S\S \ref{sec_quadratic_forms}--\ref{sec_theta_characteristics}, we refer the reader to \cite{Gro04}.

\subsection{Quadratic forms over $\bbF_2$}\label{sec_quadratic_forms}

Let $V$ be a finite-dimensional $\bbF_2$-vector space, and let $\langle \cdot, \cdot \rangle : V \times V \to \bbF_2$ be a strictly alternating pairing.
\begin{definition} A quadratic refinement of $V$ is a function $q : V \to \bbF_2$ such that for all $v,$ $w \in V$, we have $\langle v, w \rangle = q(v + w) + q(v) + q(w)$.
\end{definition}
In general, there is no distinguished quadratic refinement of $V$. However, we have the following result.
\begin{proposition} Suppose that the pairing $\langle \cdot, \cdot \rangle$ is non-degenerate.
\begin{enumerate} \item Fix a decomposition $V = U \oplus U'$, where $U,$ $U'$ are isotropic subspaces of dimension $g \geq 0$. Then the function $q_{U, U'}(v) = \langle v_U, v_{U'} \rangle$ is a quadratic refinement. \(Here we write $v_U,$ $v_{U'}$ for the projections of $v \in V$ onto the two isotropic subspaces.\)
\item The set of quadratic refinements of $V$ is a principal homogeneous space for $V$, addition being defined by the formula $(v + q)(w) = q(w) + \langle v, w \rangle$.
\end{enumerate}
\end{proposition}
\begin{definition} Suppose that the pairing $\langle \cdot, \cdot \rangle$ is non-degenerate, and let $q$ be a quadratic refinement of $V$. The Arf invariant $a(q) \in \bbF_2$ of $q$ is defined as follows. Fix a decomposition $V = U \oplus U'$ into isotropic subspaces of dimension $g \geq 0$. Let $\{ e_1, \dots, e_g \}$ be a basis of $U$, and let $\{ \epsilon_1, \dots, \epsilon_g \}$ denote the dual basis of $U'$. Then $a(q) = \sum_{i=1}^g q(e_i) q(\epsilon_i)$.
\end{definition}
\begin{lemma}\label{lem_arf_invariant} Suppose that the pairing $\langle \cdot, \cdot \rangle$ is non-degenerate, and let $\dim V = 2 g \geq 0$.
\begin{enumerate} \item The Arf invariant $a(q)$ is well-defined.
\item Let $\Sp(V)$ denote the group of automorphisms of the pair $(V, \langle \cdot, \cdot \rangle)$.  Then $\Sp(V)$ has precisely 2 orbits on the set of quadratic refinements of $V$, which are distinguished by their Arf invariants. The set of refinements with $a(q) = 0$ has cardinality $2^{g-1}(2^g + 1)$ and the set of refinements with $a(q) = 1$ has cardinality $2^{g-1}(2^g - 1)$.
\item If $q$ is a quadratic refinement and $v \in V$, then $a(q + v) = a(q) + q(v)$.
\end{enumerate}
\end{lemma}

\subsection{Theta characteristics}\label{sec_theta_characteristics}

Let $k$ be a field of characteristic 0, and let $C$ be a smooth, projective, geometrically irreducible curve over $k$, of genus $g \geq 2$. We write $K_C$ for the canonical bundle of $C$, and $J = \Pic^0(C)$ for the Jacobian of $C$. We write $V = J[2]$, a finite $k$-group. We view $V$ as an $\bbF_2$-vector space of dimension $2g$ with continuous $\Gamma_k$-action. The Weil pairing defines a non-degenerate, strictly alternating bilinear form $\langle \cdot, \cdot \rangle : V \times V \to \bbF_2$ which is $\Gamma_k$-invariant.
\begin{definition} \begin{enumerate} \item A theta characteristic is a line bundle $\cL$ on $C$ such that $\cL^{\otimes 2} \cong K_C$.
\item Let $\cL$ be a theta characteristic. We say that $\cL$ is odd \(resp. even\) if $h^0(\cL)$ is odd \(resp. even\).
\end{enumerate}
\end{definition}
Here and below we write $h^0(\cL) = \dim_k H^0(C, \cL)$ for any line bundle $\cL$ on the curve $C$.
\begin{lemma} \begin{enumerate} \item As a principal homogeneous space for $V$, the $k$-variety of isomorphism classes of theta characteristics is canonically identified with the $k$-scheme of quadratic refinements of the Weil pairing: if $\cL$ is a theta characteristic, we associate to it the quadratic refinement $q : V \to \bbF_2$ defined by the formula $q(v) = h^0(\cL \otimes_{\cO_C} v) + h^0(v) \text{ mod } 2$.
\item With notation as above, the Arf invariant of $q$ is $a(q) = h^0(\cL) \text{ mod } 2$.
\end{enumerate}
\end{lemma}
Henceforth we identify the set of theta characteristics of the curve $C$ with the set of quadratic refinements $\kappa : V \to \bbF_2$.

\subsection{Heisenberg groups and descent}\label{sec_heisenberg_and_descent}

We continue with the notation of \S \ref{sec_theta_characteristics}. Let $J^{g-1}$ denote the $J$-torsor of degree $g-1$ line bundles on $C$; it contains the theta divisor $W_{g-1}$. Given a theta characteristic $\kappa$ defined over $k$, we have the translation map $t_\kappa : J \to J^{g-1}$, $\cL \mapsto \cL \otimes \kappa$, and we define $\Theta_{\kappa} = t_\kappa^\ast W_{g-1}$. It is a symmetric divisor, and all symmetric theta divisors arise in this fashion. (This is classical; see \cite[Ch. 11]{Bir04}.) Similarly, if $A \in J(k)$ then there is a translation map $t_A : J \to J$, $\cL \mapsto \cL \otimes A$.

The isomorphism class of the line bundle $\cL_\kappa = \cO_J(2\Theta_\kappa)$ is independent of the choice of $\kappa$, but there is no canonical choice of isomorphism as $\kappa$ varies. In particular, even if $\kappa$ is defined only over $k^s$, the field of definition of this bundle is equal to $k$. We choose a bundle $\cL$ in this isomorphism class defined over $k$. We introduce the Heisenberg group $\widetilde{H}_\cL$ of pairs $(\omega, \varphi)$, where $\omega \in J[2]$ and $\varphi : \cL \to t_\omega^\ast \cL$ is an isomorphism. It is an extension
\[\xymatrix@1{ 0 \ar[r] & \bbG_m \ar[r]& \widetilde{H}_\cL \ar[r]& J[2]\ar[r]& 0.} \]
\begin{lemma}\label{lem_weil_pairing_defined_by_commutator} \begin{enumerate} \item Let $\omega,$ $\eta \in J[2]$, and let $\widetilde{\omega}$, $\widetilde{\eta}$ denote lifts of these elements to $\widetilde{H}_\cL$. Then $ \widetilde{\omega} \widetilde{\eta} \widetilde{\omega}^{-1} \widetilde{\eta}^{-1} = (-1)^{\langle \omega, \eta \rangle}$.
\item Let $\Aut(\widetilde{H}_\cL; J[2])$ denote the group of automorphisms of $\widetilde{H}_\cL$ fixing $\bbG_m$ pointwise and acting as the identity on $J[2]$. Then the map
\[ \eta \mapsto ((\omega, \varphi) \mapsto (\omega, (-1)^{\langle \eta, \omega \rangle} \varphi)) \]
defines an isomorphism $J[2] \cong \Aut(\widetilde{H}_\cL; J[2])$.
\end{enumerate}
\end{lemma}
\begin{proof}
The first part can be taken as the definition of the Weil pairing. The second part follows from \cite[Lemma 6.6.6]{Bir04}.
\end{proof}
If $\kappa$ is a theta characteristic defined over $k$, then we can define a character $\chi_\kappa : \widetilde{H}_\cL \to \bbG_m$ by the formula $\chi_\kappa(\widetilde{\omega}) = \widetilde{\omega}^2(-1)^{q_\kappa(\omega)}$. (This makes sense since the square of any element of $\widetilde{H}_\cL$ lies in $\bbG_m$.) We then have an exact sequence
\begin{equation}\label{eqn_extension_from_theta_characteristic}
\xymatrix@1{ 1 \ar[r] & \{ \pm 1 \} \ar[r] & \ker \chi_\kappa \ar[r] & J[2] \ar[r] & 1. }
\end{equation}
This construction will play an important role later on; compare the required data at the beginning of \S \ref{sec_group_with_involution} below.

Associated to $J$ is the Kummer exact sequence:
\[ \xymatrix@1{ 0 \ar[r] & J[2] \ar[r] & J \ar[r] & J \ar[r] & 0,} \]
and the associated short exact sequence in Galois cohomology:
\[ \xymatrix@1{ 0 \ar[r] & J(k)/2J(k) \ar[r]^-{\delta} & H^1(k, J[2]) \ar[r] & H^1(k, J)[2] \ar[r] & 0.}\]
The map $\delta$ can be written down explicitly as follows: given $A \in J(k)$, choose $B \in J(k^s)$ such that $[2](B) = A$. Then the cohomology class $\delta(A)$ is represented by the cocycle $\sigma \mapsto {}^\sigma B - B$.

We now give another interpretation of this homomorphism in terms of the group $\widetilde{H}_\cL$. The field of definition of the line bundle $t_B^\ast \cL$ is equal to $k$; we let $\cL_B$ denote a choice of descent to $k$, unique up to $k$-isomorphism. This allows us to define the Heisenberg group $\widetilde{H}_{\cL_B}$ of pairs $(\omega, \varphi)$, where $\omega \in J[2]$ and $\varphi$ is an isomorphism $\cL_B \to t_\omega^\ast \cL_B$. We also fix a choice of isomorphism $f : \cL_B \to t_B^\ast \cL$ over $k^s$.

The choice of $f$ defines an isomorphism $F : (\widetilde{H}_{\cL})_{k^s} \cong (\widetilde{H}_{\cL_B})_{k^s}$, given by the formula
\begin{equation}\label{eqn_isomorphism_of_heisenberg_groups} F : (\omega, \varphi) \mapsto (\omega, t_\omega^\ast f^{-1} \circ t_B^\ast \varphi \circ f).
\end{equation}
We define a cocycle valued in $\Aut(\widetilde{H}_\cL; J[2])$ by the formula $\sigma \mapsto F^{-1} {}^\sigma F$.
\begin{lemma}\label{lem_descent_via_theta_groups} This cocycle is equal to the cocycle $\sigma \mapsto {}^\sigma B - B$ under the identification of Lemma \ref{lem_weil_pairing_defined_by_commutator}.
\end{lemma}
In particular, this cocycle depends only on $B$, and not on any other choice.
\begin{proof}
The proof is by an explicit calculation, $F^{-1} {}^\sigma F$ being given by
\[ (\omega, \varphi) \mapsto \left( \omega, t^\ast_{\omega - B} f \circ t^\ast_{-B} \bigg[ t_\omega^\ast {}^\sigma f^{-1} \circ t^\ast_{{}^\sigma B} \varphi \circ {}^\sigma f \bigg] \circ t^\ast_{-B} f^{-1} \right) \]
We must show that this expression is equal to $(\omega, (-1)^{\langle \omega, {}^\sigma B - B \rangle} \varphi)$. However, writing $\eta = {}^\sigma B - B$ and $\psi = t^\ast_{-{}^\sigma B}(f \circ {}^\sigma f^{-1})$, we have $(\eta, \psi) \in \widetilde{H}_\cL$ and, by Lemma \ref{lem_weil_pairing_defined_by_commutator},
\[ (\omega, (-1)^{\langle \omega, {}^\sigma B - B \rangle} \varphi) = (\eta, \psi) (\omega, \varphi) (\eta, \psi)^{-1} (\omega, \varphi)^{-1} (\omega, \varphi) = (\eta, \psi) (\omega, \varphi) (\eta, \psi)^{-1} = (\omega,  t^\ast_{\omega + \eta} \psi \circ t_\eta^\ast \varphi \circ t_\eta^\ast \psi^{-1}).\]
Expanding this expression now shows it to be equal to $F^{-1} {}^\sigma F$.
\end{proof}

\subsection{Invariant theory of reductive groups with involution}\label{sec_vinberg_and_richardson}

Let $k$ be a field of characteristic 0, and let $H$ be a split adjoint simple group over $k$ of type $A$, $D$, or $E$.
\begin{proposition}\label{prop_stable_involution} There exists a unique $H(k)$-conjugacy class of involutions $\theta$ of $H$ satisfying the following two conditions:
\begin{enumerate} \item $\tr(d\theta : \frh \to \frh) = - \rank H$.
\item The group $(H^\theta)^\circ$ is split.
\end{enumerate}
\end{proposition}
\begin{proof}
The result \cite[Corollary 2.15]{Tho13} states that there is a unique $H(k)$-orbit of involutions $\theta : H \to H$ such that $\tr d\theta = - \rank H$ and $\frh^{d \theta = -1}$ contains a regular nilpotent element. The discussion there also shows by construction that for each $\theta$ in this class, the group $(H^\theta)^\circ$ is split. We must show that if $\theta : H \to H$ is an involution such that $\tr d \theta = - \rank H$ and $(H^\theta)^\circ$ is split, then $\frh^{d \theta = -1}$ contains a regular nilpotent. Let $\frt_0 \subset \frh^{d \theta = 1}$ be a split Cartan subalgebra, and let $\frt \subset \frh$ be a split Cartan subalgebra containing $\frt_0$.

By \cite[Lemma 2.14]{Tho13} and \cite[Lemma 2.6]{Tho13}, we can find a normal $\frs\frl_2$-triple  $(E, X, F)$ in $\frh \otimes_k k^s$ (i.e.\ a tuple of elements $E, X, F \in \frh \otimes_k k^s$ satisfying the relations
\begin{gather*}
[E, F] = X, \text{ }[X, E] = 2E, \text{ }[X, F] = -2F,\\
\theta(X) = X,\text{ }\theta(E) = -E,\text{ and }\theta(F) = -F)
\end{gather*}
with $E$ regular nilpotent and $X \in \frt_0 \otimes_k k^s$. Since $X$ is part of an $\frs\frl_2$-triple, it follows that $\alpha(X) \in \bbZ$ for every root of $\frt$ in $\frh$, hence $X \in \frt$, hence $X \in \frt_0$. By \cite[Proposition 7]{Gra11}, we can find elements $E' \in \frh^{d \theta = -1}$ and $F' \in \frh^{d \theta = -1} \otimes_k k^s$ such that $(E', X, F')$ is a normal $\frs\frl_2$-triple. In particular, $E'$ is a regular nilpotent. This completes the proof.
\end{proof}
Henceforth we fix a choice of $\theta$ satisfying the conclusion of Proposition \ref{prop_stable_involution} and write $G = (H^\theta)^\circ$. Then $G$ is a split semisimple group. (For a proof that $G$ is semisimple, see \S \ref{sec_proof_of_pro:fund} of the appendix to this paper.) We will study the invariant theory of two different actions of $G$. We first consider $V = \frh^{d \theta = -1}$. Then $V$ is a linear representation of the group $G$.
\begin{theorem}\label{thm_vinberg} \begin{enumerate} \item $V$ satisfies the Chevalley restriction theorem: if $\frt \subset V$ is a Cartan subalgebra, then the map $N_G(\frt) \to W_\frt = N_H(\frt)/Z_H(\frt)$ is surjective, and the inclusion $\frt \subset V$ induces an isomorphism
\[ \frt \!\sslash \!W_\frt \cong V \!\!\sslash\! G. \]
In particular, the quotient $V \!\!\sslash\! G$ is isomorphic to affine space.
\item Suppose that $k = k^s$, and let $x, y \in V$ be regular semisimple elements. Then $x$ is $G(k)$-conjugate to $y$ if and only if $x, y$ have the same image in $V \!\! \sslash \! G$.
\item There exists a discriminant polynomial $\Delta \in k[V]$ such that for all $x \in V$, $x$ is regular semisimple if and only if $\Delta(x) \neq 0$, if and only if the $G$-orbit of $x$ is closed in $V$ and $\Stab_G(x)$ is finite.
\end{enumerate}
\end{theorem}
\begin{proof}
This follows from results of Vinberg, which are summarized in \cite{Pan05} or (in our case of interest) \cite[\S 2]{Tho13}.
\end{proof}

We now consider the variety $Y \subset H$, locally closed image of the morphism $H \to H, h \mapsto h^{-1} \theta(h)$. It is a connected component of the subvariety $\{ h \in H \mid \theta(h) = h^{-1} \}$, and is in particular closed in $H$. Note that $Y$ has a marked point (namely the identity element of $H$), and the tangent space to $Y$ at this marked point is canonically isomorphic, as $G$-representation, to the representation $V$ defined above.
\begin{theorem}\label{thm_richardson}\begin{enumerate} \item $Y$ satisfies the Chevalley restriction theorem: if $T \subset Y$ is a maximal torus, then $N_G(T) \to W_T = N_H(T)/Z_H(T)$ is surjective, and the inclusion $T \subset Y$ induces an isomorphism
\[ T \! \sslash \! W_T \cong Y \!\! \sslash \! G. \]
\item Suppose that $k = k^s$, and let $x, y \in Y$ be regular semisimple elements. Then $x$ is $G(k)$-conjugate to $y$ if and only if $x, y$ have the same image in $Y \!\! \sslash  \!G$.
\item There exists a discriminant polynomial $\Delta \in k[Y]$ such that for all $x \in Y$, $x$ is regular semisimple if and only if $\Delta(x) \neq 0$, if and only if the $G$-orbit of $x$ is closed in $Y$ and $\Stab_G(x)$ is finite.
\end{enumerate}
\end{theorem}
\begin{proof}
See \cite[\S 0]{Ric82}.
\end{proof}

\section{A group with involution}\label{sec_group_with_involution}

Let $k$ be a field of characteristic 0. Suppose that we are given the following data:
\begin{itemize} \item An irreducible simply laced root lattice $(\Lambda, \langle \cdot, \cdot \rangle)$ together with a continuous homomorphism $\Gamma_k \to W(\Lambda) \subset \Aut(\Lambda)$.
\item A central extension $\widetilde{V}$ of $V = \Lambda / 2 \Lambda$:
\[ 0 \to \{ \pm 1 \} \to \widetilde{V} \to V \to 0, \]
together with a homomorphism $\Gamma_k \to \Aut(\widetilde{V})$. We suppose that $\Gamma_k$ leaves invariant the subgroup $\{ \pm 1 \}$, and that the induced homomorphism $\Gamma_k \to \Aut(V)$ agrees with the homomorphism $\Gamma_k \to \Aut(\Lambda) \to \Aut(\Lambda / 2 \Lambda) = \Aut(V)$. We also suppose that for $\widetilde{v} \in \widetilde{V}$, we have the relation $\widetilde{v}^2 = (-1)^{\langle v, v \rangle/2}$.
\end{itemize}
In terms of this data we will define, following Lurie \cite{Lur01}, the following:
\begin{enumerate} \item A simple Lie algebra $\frh$ over $k$ of type equal to the Dynkin type of $\Lambda$.
\item A maximal torus $T$ of $H$, the adjoint group over $k$ with Lie algebra $\frh$, together with an isomorphism $T[2](k^s) \cong V^\vee$ of $\bbZ[\Gamma_k]$-modules.
\item An involution $\theta : H \to H$ leaving $T$ stable, and satisfying $\theta(t) = t^{-1}$ for all $t \in T(k)$.
\end{enumerate}
Suppose given further the data of a finite-dimensional $k$-vector space $W$ and a homomorphism $\rho : \widetilde{V} \to \GL(W_{k^s})$ such that $\rho(-1) = - \mathrm{id}_W$ and for all $\sigma \in \Gamma_k$ and $\widetilde{v} \in \widetilde{V}$, we have $\rho({}^\sigma \widetilde{v}) = {}^\sigma \rho(\widetilde{v})$. Then we will further define:
\begin{enumerate}
\item[4:] A Lie algebra homomorphism $R : \frh^\theta \to \frg \frl(W)$.
\end{enumerate}
(Using the equivalence between $\bbZ[\Gamma_k]$-modules of finite cardinality and commutative finite $k$-groups, $\rho$ corresponds to a homomorphism $\widetilde{V} \to \GL(W)$ of $k$-groups.)

\bigskip

Let $\widetilde{\Lambda} = \Lambda \times_V \widetilde{V}$, a central extension of $\Lambda$ by $\{ \pm 1 \}$. Let $\Gamma \subset \Lambda$ be the set of roots, and $\widetilde{\Gamma} \subset \widetilde{\Lambda}$ its inverse image. Following Lurie \cite{Lur01}, we define $L'$ to be the free abelian group on symbols $X_{\widetilde{\gamma}}$ for $\widetilde{\gamma} \in \widetilde{\Gamma}$, modulo the relation $X_{\widetilde{\gamma}} = - X_{-\widetilde{\gamma}}$. (Thus $\{ \widetilde{\gamma}, -\widetilde{\gamma} \}$ is the inverse image in $\widetilde{\Gamma}$ of $\gamma \in \Gamma$.) We set $L = \Lambda^\vee \oplus L'$, and define a bracket $[ \cdot, \cdot ] : L \times L \to L$ by the formulae:
\begin{itemize} \item $[\lambda, \lambda'] = 0$ for all $\lambda,$ $\lambda' \in \Lambda^\vee$.
\item $[\lambda, X_{\widetilde{\gamma}}] = -[X_{\widetilde{\gamma}}, \lambda] = \langle \lambda, \gamma \rangle X_{\widetilde{\gamma}}$ for $\lambda \in \Lambda^\vee$.
\item $[X_{\widetilde{\gamma}}, X_{\widetilde{\gamma}'}] = X_{\widetilde{\gamma}\widetilde{\gamma}'}$ if $\gamma + \gamma' \in \Gamma$.
\item $[X_{\widetilde{\gamma}}, X_{\widetilde{\gamma}'}] = \epsilon_{\widetilde{\gamma}\widetilde{\gamma}'} \gamma$ if $\gamma + \gamma' = 0$. (By definition, $\epsilon_{\widetilde{\gamma}\widetilde{\gamma}'} = \widetilde{\gamma}\widetilde{\gamma}' \in \{ \pm 1\} \subset \bbZ$.)
\item $[X_{\widetilde{\gamma}}, X_{\widetilde{\gamma}'}] = 0$ otherwise.
\end{itemize}

\begin{theorem} \begin{enumerate} \item $L$ is a Lie algebra over $\bbZ$. There is a natural action of $\Gamma_k$ on $L$, respecting the Lie bracket $[ \cdot, \cdot ]$.
\item Let $\frh = (L \otimes_k k^s)^{\Gamma_k}$. Then $\frh$ is a simple Lie algebra over $k$ of Dynkin type equal to the type of the root lattice $\Lambda$.
\end{enumerate}
\end{theorem}
\begin{proof} \begin{enumerate} \item That $L$ is a Lie algebra over $\bbZ$ of the required type follows from \cite[\S 3.1]{Lur01}. The Galois group $\Gamma_k$ acts on $\Lambda$ and on $\widetilde{\Gamma}$ by the given data. We make it act on $L = \Lambda \oplus L'$ by its standard action on $\Lambda$ and on $L'$ by permuting basis vectors $X_{\widetilde{\gamma}},$ $\widetilde{\gamma} \in \widetilde{\Gamma}$. It is immediate from the definition that this respects the bracket.
\item By Galois descent, the natural map $\frh_{k^s} \to L \otimes_k k^s$ is an isomorphism. The result follows immediately from this.
\end{enumerate}
\end{proof}
Let $H$ denote the simple adjoint group over $k$ with Lie algebra $\frh$. Let $\frt = (\Lambda^\vee \otimes_k k^s)^{\Gamma_k} \subset \frh$; it is the Lie algebra of a maximal torus $T$ of $H$, whose module of characters $X^\ast(T_{k^s})$ is identified with the $\bbZ[\Gamma_k]$-module $\Lambda$. In particular, there is an isomorphism of $\bbZ[\Gamma_k]$-modules $T[2](k^s) \cong \Lambda^\vee / 2 \Lambda^\vee \cong V^\vee$.

We now define the involution $\theta$. Given $\widetilde{\gamma} \in \widetilde{\Gamma}$, we define $Y_{\widetilde{\gamma}} = X_{\widetilde{\gamma}^{-1}}$. By definition, then, $[X_{\widetilde{\gamma}}, Y_{\widetilde{\gamma}}] = \gamma \in \Lambda$. It easy to check that $Y_{-\widetilde{\gamma}} = -Y_{\widetilde{\gamma}}$. We define an involution $\sigma : L \to L$ by taking $\sigma$ to be multiplication by $-1$ on $\Lambda$ and by taking $\sigma(X_{\widetilde{\gamma}}) = -Y_{\widetilde{\gamma}}$.
\begin{proposition} \begin{enumerate} \item $\sigma$ is a well-defined Lie algebra involution, and respects the action of the group $\Gamma_k$.
\item Let $\theta$ denote the involution of $\frh$ induced by $\sigma$ by functoriality. Then $\tr \theta = - \rank \frh$.
\end{enumerate}
\end{proposition}
\begin{proof}
\begin{enumerate} \item We must check that $\sigma$ preserves the relations defining $[\cdot, \cdot]$. Let us show that $\sigma[X_{\widetilde{\gamma}}, X_{\widetilde{\gamma}'}] = \sigma X_{\widetilde{\gamma}\widetilde{\gamma}'} = - Y_{\widetilde{\gamma} \widetilde{\gamma}'}$ is equal to $[\sigma X_{\widetilde{\gamma}}, \sigma X_{\widetilde{\gamma}'}] = [X_{\widetilde{\gamma}^{-1}}, X_{\widetilde{\gamma}'^{-1}}] = X_{\widetilde{\gamma}^{-1} \widetilde{\gamma}'^{-1}}$, when $\gamma + \gamma' \in \Gamma$. Equivalently, we must show that $\widetilde{\gamma} \widetilde{\gamma}'= - \widetilde{\gamma}' \widetilde{\gamma}$. By the definition of $\widetilde{\Lambda}$, it is equivalent to show that $\langle \gamma, \gamma' \rangle$ is odd. Since we work in a simply laced root system, this is implied by the condition that $\gamma + \gamma'$ is a root.
\item This follows because $\theta$ acts as $-1$ on $\frt$.
\end{enumerate}
\end{proof}
We define $G = (H^\theta)^\circ$. We define $N_V$ to be the image of the natural homomorphism $V \to V^\vee$; it is a $\bbZ[\Gamma_k]$-module, and the induced symplectic form on $N_V$ is non-degenerate and $\Gamma_k$-equivariant. The isomorphism $T[2] \cong V^\vee$ restricts to an isomorphism $(T[2] \cap G)\cong N_V$ (cf. \cite[Corollary 2.8]{Tho13}).

It remains to define, given a finite-dimensional $k$-vector space $W$ and a Galois-equivariant homomorphism $\rho : \widetilde{V} \to \GL(W_{k^s})$ such that $\rho(-1) = - \mathrm{id}_W$, a Lie algebra homomorphism $R : \frg \to \frg \frl(W)$. Let us first assume that $k = k^s$. Then the Lie algebra $\frg$ is spanned by the elements $X_{\widetilde{\gamma}} + X_{-\widetilde{\gamma}^{-1}} = Z_{\widetilde{\gamma}}$, say. Let $\pi : \widetilde{\Gamma} \to \widetilde{V}$ denote the natural map. We define a morphism $R : \frg \to \frg\frl(W)$ of $k$-vector spaces by the formula
\[ R(Z_{\widetilde{\gamma}}) = \rho(\pi(\widetilde{\gamma}))/2. \]
This is well-defined since $Z_{\widetilde{\gamma}} = -Z_{-\widetilde{\gamma}} = -Z_{\widetilde{\gamma}^{-1}}$, and  $\pi(\widetilde{\gamma}) = (-1)^{\langle \gamma, \gamma \rangle/2} \pi(\widetilde{\gamma})^{-1} = - \pi(\widetilde{\gamma})^{-1}$. In the case $k \neq k^s$, this defines a homomorphism $\frg_{k^s} \to \frg\frl(W_{k^s})$ which commutes with the action of $\Gamma_k$, and we write $R : \frg \to \frg\frl(W)$ for the homomorphism obtained by Galois descent.
\begin{proposition} $R : \frg \to \frg \frl(W)$ is a Lie algebra homomorphism. \label{pro:lie_alg_rep}
\end{proposition}
\begin{proof} We can again assume that $k = k^s$. We must show that, given $\widetilde{\gamma},$ $\widetilde{\gamma}' \in \widetilde{\Gamma}$, we have
\[ R([Z_{\widetilde{\gamma}}, Z_{\widetilde{\gamma}'}]) = [R(Z_{\widetilde{\gamma}}), R(Z_{\widetilde{\gamma}'})]. \]
We now break up into cases according to the value of $\langle \gamma, \gamma' \rangle$.
\begin{enumerate} \item If $\langle \gamma, \gamma' \rangle = \pm 2$, then $\gamma' = \pm \gamma$, hence $\widetilde{\gamma}' = \pm \widetilde{\gamma}^{\pm1}$, and both sides of the above equation are zero.
\item If $\langle \gamma, \gamma' \rangle = \pm 1$, then $\gamma \mp \gamma'$ is a root. Let us assume for simplicity that $\langle \gamma, \gamma' \rangle = -1$, so that $\gamma + \gamma'$ is a root, and $[ Z_{\widetilde{\gamma}}, Z_{\widetilde{\gamma}'}] = Z_{\widetilde{\gamma} \widetilde{\gamma}'}$. We must show that
\[ \rho(\pi(\widetilde{\gamma}\widetilde{\gamma}'))/2 = \rho(\pi(\widetilde{\gamma})) \cdot \rho(\pi(\widetilde{\gamma}'))/4 - \rho(\pi(\widetilde{\gamma}')) \cdot \rho(\pi(\widetilde{\gamma}))/4. \]
This follows from the fact that $\widetilde{\gamma}' \widetilde{\gamma} = (-1)^{\langle \gamma, \gamma' \rangle} \widetilde{\gamma} \widetilde{\gamma}' = - \widetilde{\gamma} \widetilde{\gamma}'$ and $\rho(-1) = -\mathrm{id}_W$.
\item If $\langle \gamma, \gamma' \rangle = 0$ then neither of $\gamma \pm \gamma'$ is a root, and the left hand side of the above equation is zero. On the other hand, $\pi(\widetilde{\gamma})$ and $\pi(\widetilde{\gamma}')$ commute, so the right hand side is also zero.
\end{enumerate}
This concludes the proof.
\end{proof}

The above constructions are evidently functorial in $\widetilde{V}$, in the following sense: given $\widetilde{V}$, $\widetilde{V}_B$ satisfying the conditions at the beginning of this section, and a $\Gamma_k$-equivariant isomorphism $f : \widetilde{V} \to \widetilde{V}_B$ , we obtain an isomorphism of associated simple adjoint groups $F : H \cong H_B$, intertwining $\theta,$ $\theta_B$, and restricting to an isomorphism $T \to T_B$ which induces the identity on $\Lambda$. In this connection, we have the following lemma.

\begin{lemma}\label{lem_functoriality_of_isomorphisms} \begin{enumerate} \item Let us write $\Aut(\widetilde{V}; V)$ for the group of automorphisms of $\widetilde{V}$ leaving the central subgroup $\{ \pm 1 \}$ invariant and inducing the identity on $V$. Then there is a canonical isomorphism $V^\vee \cong \Aut(\widetilde{V}; V)$, given by $f \mapsto (\widetilde{v} \mapsto (-1)^{f(v)} \cdot \widetilde{v})$.
\item Let $f \in V^\vee$, and let $F$ denote the induced automorphism of the triple $(H, \theta, T)$. Let $s$ denote the image of $f$ under the canonical isomorphism $V^\vee \cong T[2](k^s)$. Then $F = \Ad(s)$.
\end{enumerate}
\end{lemma}
\begin{proof} \begin{enumerate} \item Immediate.
\item The automorphism $f$ induces the automorphism $\widetilde{\gamma} \mapsto (-1)^{f(\gamma)}\widetilde{\gamma}$ of $\widetilde{\Gamma}$. We must therefore show that $(-1)^{f(\gamma)} = \langle \gamma, s \rangle$. However, this follows from the definition of the element $s$.
\end{enumerate}
\end{proof}

\section{Plane quartic curves}\label{sec_plane_quartic_curves}

Let $k$ be a field of characteristic 0 and $C$ a smooth (geometrically connected, projective) non-hyperelliptic curve of genus 3 over $k$. The canonical embedding then gives $C$ as a plane quartic curve in $\bbP^2_k$; let us write $\pi : S \to \bbP^2_k$ for the double cover of $\bbP^2_k$ branched over $S$. Then $S$ is a del Pezzo surface of degree 2, i.e.\ a smooth surface with $-K_S$ ample and $K_S^2 = 2$. (We note that if $k \neq k^s$, then $S$ depends, up to isomorphism, on a choice of defining equation of $C$; a particular choice will be specified below. The set of isomorphism classes is a torsor for $k^\times/(k^\times)^2$.)

\begin{proposition}\label{prop_geometric_root_lattice}
\begin{enumerate} \item The group $\Pic(S_{k^s})$ is free of rank 8 over $\bbZ$. Its natural intersection pairing is unimodular.
\item The sublattice $\Lambda = K_S^\perp \subset \Pic(S_{k^s})$ is a root lattice of type $E_7$.
\item Suppose that $\ell$ is a bitangent line of $C$ in its canonical embedding. Then $\pi^{-1}(\ell_{k^s}) = e \cup f$ is a union of two smooth curves of genus 0. Define $\Lambda_\ell = \langle e, f \rangle^\perp \subset \Lambda$. Then $\Lambda_\ell$ is a root lattice of type $E_6$.
\item There are natural isomorphisms $\Lambda^\vee \cong \Pic(S_{k^s})/\bbZ K_S$ and $\Lambda_\ell^\vee \cong \Pic(S_{k^s})/\langle e, f \rangle$.
\item Each of $\Pic(S_{k^s})$, $\Lambda$, and $\Lambda_\ell$ (when it is defined) has a natural structure of $\bbZ[\Gamma_k]$-module, which respects the intersection pairings.
\end{enumerate}
\end{proposition}
\begin{proof}
This is all classical; see \cite[pp. 545--549]{Gri94} and \cite[Ch. 8]{Dol12}. It is useful to note that $S_{k^s}$ can be realized as the blow-up of $\bbP^2_{k^s}$ at 7 points in general position.
\end{proof}
We define $N_C$ to be the image of the natural map $\Lambda/2\Lambda \to \Lambda^\vee / 2 \Lambda^\vee$. Viewing $C \subset S$ as the ramification locus of $\pi$, we see that there is a natural $\Gamma_k$-equivariant map $\Pic(S_{k^s}) \to \Pic(C_{k^s})$ given by restriction of line bundles.
\begin{proposition}\label{prop_plane_quartic_commutative_diagram} There is a commutative diagram of finite $k$-groups
\[ \xymatrix{ \Lambda^\vee/2\Lambda^\vee \ar[r]^-{\cong} & (\Pic(C)/\bbZ K_C)[2] \\
N_C \ar@{^{(}->}[u] \ar[r]^-{\cong} & \Pic^0(C)[2]\ar@{^{(}->}[u].} \]
\end{proposition}
\begin{proof}
We first define the maps. The top map is induced by the composite
\[ \Lambda^\vee \cong \Pic(S)/\bbZ K_S \to  \Pic(C)/\bbZ K_C, \]
which takes image in $(\Pic(C)/\bbZ K_C)[2] \subset \Pic(C)/\bbZ K_C$. It is well-defined since $K_S|_{C} = -K_C$, and if $D$ is any divisor class on $S$ then $2D|_C \sim (D + \iota^\ast D)|_C$ is a multiple of $K_C$ (where $\iota : S \to S$ is the involution which swaps sheets). The left and right maps are the natural inclusions. To see that the bottom map is derived from the top one, it is enough to note that if $D$ is a divisor class in $\Lambda$, then $\deg D|_C = \langle K_S, D \rangle = 0$, so $D|_C \in \Pic^0(C)[2]$.

We now show that the top and bottom maps are isomorphisms. We can assume that $k = k^s$. The groups in the top row have the same cardinality $2^7$. If $\ell$ is a bitangent line of $C$ corresponding to an odd theta characteristic $\kappa \in (\Pic(C)/\bbZ K_C)[2]$, and $\pi^{-1}(\ell) = e \cup f$, then the image of $e \in \Lambda^\vee$ in $(\Pic(C)/\bbZ K_C)[2]$ equals $\kappa$. The group $(\Pic(C)/\bbZ K_C)[2]$ is generated by the odd theta characteristics. This shows that the top arrow is surjective, hence an isomorphism. The groups in the bottom row have the same cardinality $2^6$, and the bottom arrow is injective. It is therefore also an isomorphism, and this completes the proof.
\end{proof}
As pointed out in the introduction, Proposition \ref{prop_plane_quartic_commutative_diagram} is essentially classical.
\begin{proposition}\label{prop_symplectic_pairings_are_identified} \begin{enumerate} \item Under the isomorphism $N_C \cong \Pic^0(C)[2]$ of Proposition \ref{prop_plane_quartic_commutative_diagram}, the natural symplectic form on $N_C$ is identified with the Weil pairing on $\Pic^0(C)[2]$.
\item Let $\ell$ be a $k$-rational bitangent line of $C$, and let $\kappa$ denote the corresponding $k$-rational theta characteristic. Let $q_\ell : N_C \to \bbF_2$ denote the quadratic form corresponding to the isomorphism $\Lambda_\ell / 2 \Lambda_\ell \cong N_C$, and let $q_\kappa : \Pic^0(C)[2] \to \bbF_2$ be the quadratic form induced by $\kappa$. Then, under the isomorphism $N_C \cong \Pic^0(C)[2]$ of Proposition \ref{prop_plane_quartic_commutative_diagram}, $q_\ell$ and $q_\kappa$ are identified.
\end{enumerate}
\end{proposition}
\begin{proof}
Since $q_\ell$ and $q_\kappa$ are quadratic refinements of the symplectic forms, it suffices to prove the second part. These quadratic forms have Arf invariant $1$, and therefore have each exactly 28 zeroes. It therefore suffices to show that $q_\ell$ and $q_\kappa$ have at least 28 zeroes in common. To do this, we can assume that $k = k^s$. If $\kappa'$ is any odd theta characteristic of $C$, then $\kappa - \kappa' \in \Pic^0(C)[2]$ is a zero of $q_\kappa$, and there are exactly 28 such elements. (Use the formula $a(q+v) = a(q) + q(v)$ of Lemma \ref{lem_arf_invariant}.) We must therefore show that if $v \in \Lambda_\ell$ has image $\kappa - \kappa'$, then $\langle v, v \rangle$ is divisible by 4. This is an easy calculation in $\Pic(S_{k^s})$.
\end{proof}
We now fix a rational point $P \in C(k)$. We define elements of certain tori and their Lie algebras, following \cite[\S 1]{Loo93}. We break into 4 cases, according to the geometry of the point $P$.  Let $\ell$ denote the tangent line to $C$ at $P$ in $\bbP^2_k$, and $K = \pi^{-1}(\ell)$ its inverse image, an anti-canonical curve in $S$.
 \subsubsection*{Case $E_7$: $\ell$ not a flex}
In the most general case, the tangent line at $P$ to $C$ in its plane embedding meets $C$ at 3 distinct points and therefore has contact of order 2 at $P$. We define a point of the torus $T = \Hom(\Lambda, \bbG_m)$, up to inversion.
Indeed, in this case $K$ is an irreducible rational curve with a unique nodal singularity at $P$. There is a unique choice of $S$ for which the tangent directions of $K$ at $P$ are defined over $k$; we make this choice. Restriction of line bundles induces a homomorphism $\Pic(S) \to \Pic(K)$. An element of $\Pic(S)$ is orthogonal to $K_S$ (under the intersection pairing) if and only if its restriction to $K$ has degree 0, so we obtain an induced homomorphism $\Lambda \to \Pic^0(K)$. Choosing a group isomorphism $\Pic^0(K) \cong \bbG_m$, we now obtain a point $\kappa_C \in T(k)$, well-defined up to inversion.

\subsubsection*{Case $\fre_7$: $\ell$ a flex, not a hyperflex}
We now suppose that the tangent line to $C$ at $P$ has contact of order exactly 3, and fix in addition a non-zero tangent vector $t$ in the Zariski tangent space of $C$ at $P$. We define a point $\kappa_C$ of the Lie algebra $\frt$ of the torus $T = \Hom(\Lambda, \bbG_m)$, well-defined up to multiplication by $-1$. The curve $K$ is irreducible and rational with a unique cuspidal singularity, at $P$. Restriction induces a morphism $\Lambda \to \Pic^0(K)$. To write down $\kappa_C$, it therefore suffices to give a normalization of the isomorphism $\Pic^0(K) \cong \bbG_a$, at least up to sign.

To do this we find it convenient to introduce explicit co-ordinates. Using Riemann--Roch, it is easy to show that there are unique functions $x, y \in k(C)^\times$ satisfying the following conditions:
\begin{itemize}
\item $x \in H^0(C, \cO_C(2P + Q))$ and $y \in H^0(C, \cO_C(3P - Q))$.
\item Let $z \in \cO_{C, P}$ be a co-ordinate such that $dz(t) = 1$. Then $x = z^{-2} + \dots$ and $y = z^{-3} + \dots$ locally at $P$.
\item $x$ and $y$ satisfy the equation
\[ y^3 = x^3y + p_{10}x^2 + x(p_2y^2 + p_8y + p_{14}) + p_6y^2 + p_{12}y + p_{18} \]
for some $p_2, \dots, p_{18} \in k$.
\end{itemize}
Then we can choose homogeneous co-ordinates $X, Y, Z$ on $\bbP^2_k$ such that $C$ is given by the equation
\[ Y^3 Z = X^3 Y + p_{10}X^2Z^2 + X(p_2Y^2Z + p_8YZ^2 + p_{14}Z^3) + p_6Y^2Z^2 + p_{12}YZ^3 + p_{18}Z^4, \]
and this equation is uniquely determined by the triple $(C, P, t)$. We use it to define the surface $S$. Then a chart in $S$ is the affine surface
\[ w^2 = z_0 - ( x_0^3 + p_{10} x_0^2 z_0^2 + \dots + p_{18} z_0^4), \]
where $x_0 = X/Y$, $z_0 = Z/Y$, and the curve $K$ is given locally by the equation $z_0 = 0$. Let $f : \widetilde{K} \to K$ be the normalization. A co-ordinate in $\widetilde{K}$ at the point above $P$ is given by $w/x_0$. We use the isomorphism $\bbG_a \cong \Pic^0(K)$, $t \mapsto \delta(1 + tw/x_0)$, where $\delta$ is the connecting homomorphism of the exact sequence of sheaves on $K$:
\[ \xymatrix@1{ 0 \ar[r] & \cO_K^\times \ar[r] & f_\ast \cO_{\widetilde{K}}^\times \ar[r] & f_\ast \cO_{\widetilde{K}}^\times/\cO_K^\times \ar[r] & 0.} \]

\subsubsection*{Case $E_6$: $\ell$ a bitangent, not a hyperflex}
We now suppose that $\ell$ meets $C$ at two distinct points, say $P$, $Q$, and that it has contact of order 2 at each point. Then the root subsystem $\Lambda_\ell \subset \Lambda$ is defined, and we will define a point of the torus $T = \Hom(\Lambda_\ell, \bbG_m)$. The curve $K_{k^s} = e_{k^s} \cup f_{k^s}$ is a union of two smooth conics, which meet transversely at two distinct points. We choose $S$ so that these conics are defined over $k$. We thus obtain a homomorphism $\Lambda_\ell \to \Pic^0(K)^-$, where $(?)^-$ denotes the $-1$-eigenspace of the involution induced by switching sheets. The group $\Pic^0(K)^-$ is canonically isomorphic to $\bbG_m$, the isomorphism being specified as in \cite[\S 1.12]{Loo93}: if $s \in \bbG_m$ tends to 0, then $e$ is contracted to $P$ and $f$ is contracted to $Q$. We define $\kappa_C \in T(k)$ to be the point obtained via this isomorphism. If the roles of $e$ and $f$ are reversed, then $\kappa_C$ is replaced by $\kappa_C^{-1}$.

\subsubsection*{Case $\fre_6$: $\ell$ a hyperflex}
We now suppose that $\ell$ has contact of order $4$ with $C$ at $P$, and fix in addition a non-zero tangent vector $t$ in the Zariski tangent space of $C$ at $P$. Then the root system $\Lambda_\ell \subset \Lambda$ is defined, and we will define a point $\kappa_C$ of the Lie algebra $\frt$ of the torus $T = \Hom(\Lambda_\ell, \bbG_m)$. Restriction once more induces a map $\Lambda_\ell \to \Pic^0(K)^-$, and we obtain a point $\kappa_C \in \frt$ by specifying an isomorphism $\Pic^0(K)^- \cong \bbG_a$. To do this, we again introduce explicit co-ordinates. There are unique functions $x, y \in k(C)^\times$ satisfying the following conditions:
\begin{itemize}
\item $x \in H^0(C, \cO_C(3P))$ and $y \in H^0(C, \cO_C(4P))$.
\item Let $z \in \cO_{C, P}$ be a co-ordinate such that $dz(t) = 1$. Then $x = z^{-3} + \dots$ and $y = z^{-4} + \dots$ locally at $P$.
\item $x$ and $y$ satisfy the equation
\[ y^3 = x^4 + y(p_2x^2 + p_5x + p_8) + p_6x^2 + p_9x + p_{12} \]
for some $p_2, \dots, p_{12} \in k$.
\end{itemize}
Then we can choose homogeneous co-ordinates $X, Y, Z$ on $\bbP^2_k$ such that $C$ is given by the equation
\[ Y^3Z = X^4 + Y(p_2X^2Z + p_5XZ^2 + p_8Z^3) + p_6X^2Z^2 + p_9XZ^3 + p_{12}Z^4, \]
and this equation is uniquely determined by the triple $(C, P, t)$. We use it to define the surface $S$. A chart in $S$ is the affine surface
\[ w^2 = z_0 - (x_0^4 + \dots + p_{12} z_0^4), \]
where $x_0 = X/Y$ and $z_0 = Z/Y$. The curve $K = e \cup f$ is a union of 2 smooth conics which are tangent at the point $P$, and is given in the above chart by the equation $z_0 = 0$. A co-ordinate at $P$ in both $e$ and $f$ is given by $x_0$. We use the isomorphism $\bbG_a \cong \Pic^0(K)^-$, $t \mapsto \delta(1 + t x, 1)$, where $\delta$ is the connecting homomorphism in the exact sequence of sheaves on $K$:
\[ \xymatrix@1{ 0 \ar[r] & \cO_K^\times \ar[r] & \cO_e^\times \oplus \cO_f^\times \ar[r] & (\cO_e^\times \oplus \cO_f^\times)/\cO_K^\times \ar[r] & 0.} \]
If the roles of $e$ and $f$ are reversed, then $\kappa_C$ is replaced by $-\kappa_C$.

\bigskip

In each case, we write $\cS : k\text{-alg} \to \text{Sets}$ for the functor of data $(C, P, \dots)$ considered above. This means:
\begin{itemize}
\item In case $E_7$, $\cS$ is the functor of pairs $(C, P)$, where $C$ is a non-hyperelliptic curve of genus 3 and $P$ is a point of $C$ which is not a flex or a bitangent in the canonical embedding. More precisely, for each $A \in k\text{-alg}$, $\cS(A)$ is the set of isomorphism classes of pairs $(\pi, P)$ consisting of a proper flat morphism $\pi : C \to \Spec A$ and a section $P : \Spec A \to C$ of $\pi$ such that for each geometric point $\overline{s}$ of $\Spec A$, the pair $(C_{\overline{s}}, P_{\overline{s}})$ is of this type.
\item In case $\fre_7$, $\cS$ is the functor of triples $(C, P, t)$, where $C$ is a non-hyperelliptic curve of genus 3, $P$ is a point of $C$ which is a flex (but not a hyperflex) in the canonical embedding, and $t$ is a non-zero element of the Zariski tangent space of $C$ at $P$.
\item In case $E_6$, $\cS$ is the functor of pairs $(C, P)$, where $C$ is a non-hyperelliptic curve of genus 3 and $P$ is a point such that $T_PC$ is a bitangent in the canonical embedding of $C$.
\item In case $\fre_6$, $\cS$ is the functor of triples $(C, P, t)$, where $C$ is a non-hyperelliptic curve of genus 3, $P$ is a point which is a hyperflex in the canonical embedding, and $t$ is a non-zero element of the Zariski tangent space of $C$ at $P$.
\end{itemize}
We can now state the following reformulation of some results of Looijenga:
\begin{theorem}\label{thm_application_of_looijenga}
Suppose that $k = k^s$.
\begin{itemize}
\item In case $E_7$, let $\Lambda_0$ be a root lattice of the corresponding type, and let $T_0 = \Hom(\Lambda_0, \bbG_m)$. Then the Weyl group $W = W(\Lambda_0)$ acts on $T_0$, and the assignment $(C, P) \to \kappa_C$ induces a bijection $\cS(k) \to (T_0^\text{rss} \! \sslash \!  W)(k)$.
\item In case $E_6$, let $\Lambda_0$ be a root lattice of the corresponding type, and let $T_0 = \Hom(\Lambda_0, \bbG_m)$. Fix a non-trivial class $e_0 \in \Lambda_0^\vee / \Lambda_0$. Then the Weyl group $W = W(\Lambda_0)$ acts on $T_0$, and the assignment $(C, P) \to \kappa_C$ induces a bijection $\cS(k) \to  (T_0^\text{rss} \! \sslash \!  W)(k)$.
\item In case $\fre_7$, let $\Lambda_0$ be a root lattice of the corresponding type, and let $\frt_0 = \Hom(\Lambda_0, \bbG_a)$. Then the Weyl group $W = W(\Lambda_0)$ acts on $\frt_0$, and the assignment $(C, P, t) \to \kappa_C$ induces a bijection $\cS(k) \to (\frt_0^\text{rss} \! \sslash \!  W)(k)$.
\item In case $\fre_6$, let $\Lambda_0$ be a root lattice of the corresponding type, and let $\frt_0 = \Hom(\Lambda_0, \bbG_a)$. Fix a non-trivial class $e_0 \in \Lambda_0^\vee / \Lambda_0$. Then the Weyl group $W = W(\Lambda_0)$ acts on $\frt_0$, and the assignment $(C, P, t) \to \kappa_C$ induces a bijection $\cS(k) \to (\frt_0^\text{rss} \! \sslash \!  W)(k)$.
\end{itemize}
\end{theorem}
The subscript `rss' indicates the open subset of regular semisimple elements, i.e.\ the complement of all root hyperplanes.
\begin{proof}
We first explain what happens in the case of type $E_7$. For any field $k$ (not necessarily separably closed), and any pair $(C, P) \in \cS(k)$, we have constructed a point $\kappa_C$ of the torus $T = \Hom(\Lambda, \bbG_m)$, where $\Lambda$ is the root lattice with $\bbZ[\Gamma_k]$-action constructed above using the curve $C$.

When $k = k^s$, this action is trivial, and we can choose an isomorphism $\Lambda \cong \Lambda_0$ of root lattices, which is well-defined up to the action of the group $\Aut(\Lambda_0)$. The Dynkin diagram of type $E_7$ has no extra symmetries, so in fact $\Aut(\Lambda_0) = W$ (see \cite[Ch. VI, No. 1.5, Proposition 16]{Bou02}). We thus obtain an isomorphism $T \cong T_0$, well-defined up to the action of $W$, and a point $\kappa_C \in (T_0 \! \sslash \!  W )(k) = T_0(k) / W$. Note that $\kappa_C$ is well-defined only up to inversion, but $W$ contains the element $-1$. The result \cite[Proposition 1.8]{Loo93} now states that the point $\kappa_C$ is regular semisimple, and that the map $\cS(k) \to (T_0^\text{rss} \! \sslash \!  W)(k)$ is a bijection. (In fact, the result is stated when $k = \bbC$, but the proof is algebro-geometric in nature and goes through without change when $k$ is any separably closed field of characteristic 0.) Indeed, the construction given there is exactly the one we have explicated above.

We now explain what happens in the case of type $E_6$. Our construction gives a point $\kappa_C = \kappa(C, P, e)$ of the torus $T = \Hom(\Lambda_\ell, \bbG_m)$, where $e$ is a choice of irreducible component of the strict transform of the bitangent line $\ell$ at $P$ inside $S$; we have $\kappa(C, P, f) = \kappa(C, P, e)^{-1}$. The automorphism group $\Aut(\Lambda_0)$ is now strictly larger than $W$, because the Dynkin diagram of type $E_6$ has extra symmetries, the quotient $\Aut(\Lambda_0) / W$ being generated by the automorphism $-1$. In fact, these ambiguities cancel out.

Indeed, the quotient $\Lambda_\ell^\vee / \Lambda_\ell$ is cyclic of order 3, and the quotient $\Aut(\Lambda_0) / W$ acts faithfully on it. We can mark the non-trivial elements of $\Lambda_\ell^\vee / \Lambda_\ell$ by $e$ and $f$ as follows: the class corresponding to $e$ is the one containing the classes of the 27 lines on $S$ which intersect $e$ (but not $f$), and the class corresponding to $f$ is the one containing the classes of the 27 lines which intersect $f$ (but not $e$). Let $\lambda_e : \Lambda_\ell \to \Lambda_0$ be an isomorphism which sends the class in $\Lambda_\ell^\vee / \Lambda_\ell$ corresponding to $e$ to $e_0$. Then $\lambda_e$ is determined up to the action of $W(\Lambda_0)$. The point $\lambda_e \kappa(C, P, e) \in (T_0 \! \sslash \!  W)(k)$ is therefore well-defined, and we have $\lambda_f \kappa(C, P, f) = (\lambda_e \kappa(C, P, e)^{-1})^{-1} = \lambda_e \kappa(C, P, e) \text{ mod } W_0$. This gives a map $\cS(k) \to (T_0 \! \sslash \!  W)(k)$ which is independent of any choices, and which is shown to be a bijection into $(T_0^\text{rss}\! \sslash \!  W)(k)$ by \cite[Proposition 1.13]{Loo93}.

The Lie algebra cases are very similar, making reference to \cite[Proposition 1.11]{Loo93} and \cite[Proposition 1.15]{Loo93}.
\end{proof}
\subsection{Construction of orbits}\label{sec_construction_of_orbits}
We now come to the most important part of this paper. In each of the cases $E_7$, $\fre_7$, $E_6$ and $\fre_6$ described above, we give a semisimple group $G$ over $k$, together with a $G$-variety $X$, and write down orbits in $G(k) \backslash X(k)$ corresponding to elements of the groups $J(k) / 2 J(k)$. We must first fix `reference data'. This means:
\begin{itemize}
\item In cases $E_7$ and $\fre_7$, we fix a choice of pair $(H, \theta)$, where $H$ is a split adjoint simple group over $k$ of type $E_7$, and $\theta$ is an involution satisfying the conditions of Proposition \ref{prop_stable_involution}. We define $G = (H^\theta)^\circ$, and fix an inner class of isomorphisms $\frg \cong \frs\frl_8$; equivalently, we distinguish one of the two 8-dimensional representations of $\frg$ as the `standard representation'. The group $H$ has no outer automorphisms, but the group $H^\theta$ has two connected components, and the non-identity component acts on the identity component $G$ by outer automorphisms, exchanging the two choices of standard representation. Indeed, the component group can be calculated using \cite[Proposition 2.1]{Ree10} and the Kac co-ordinates of the inner automorphism $\theta$, which appear in the tables in \cite{Gro12}. The proof of \cite[Proposition 2.1]{Ree10} shows that we can find a representative of the non-trivial component which normalizes a maximal torus of $G$ but which does not act on this torus in the same way as any Weyl element of $G$; the induced automorphism of $G$ must therefore be outer.
\item In cases $E_6$ and $\fre_6$, we fix a choice of pair $(H, \theta)$, where $H$ is a split adjoint simple group over $k$ of type $E_6$, and $\theta$ is an involution satisfying the conditions of Proposition \ref{prop_stable_involution}. We define $G = (H^\theta)^\circ = H^\theta$, and distinguish one of the two 27-dimensional representations of $\frh$ as the `standard representation'. The connectedness of $H^\theta$ can be shown as above using the papers \cite{Ree10, Gro12}.
\end{itemize}
We recall that in \S \ref{sec_vinberg_and_richardson} we have defined two $G$-varieties $Y$ and $V$ in terms of the pair $(H, \theta)$. We use these to define the $G$-variety $X$ as follows:
\begin{itemize}
\item In cases $E_7$ and $E_6$, we define $X = Y \subset H$.
\item In cases $\fre_7$ and $\fre_6$, we define $X = V \subset \frh$.
\end{itemize}
In each case there is a $G$-invariant open subscheme $X^\text{s} \subset X$ of regular semisimple (equivalently, stable) orbits. We can now state our first main theorem:
\begin{theorem}\label{thm_base_orbits}
In each case, the assignment $(C, P, \dots) \mapsto \kappa_C$ determines a map
\begin{equation}\label{eqn_base_orbit} \cS(k) \to G(k) \backslash X^\text{s}(k).
\end{equation}
If $k = k^s$, then this map is bijective.
\end{theorem}
We observe that the theorem has already been proved in the case $k = k^s$. Indeed, in this case, the set $G(k) \backslash X^s(k)$ can be understood, via the Chevalley isomorphisms of \S \ref{sec_vinberg_and_richardson}, in terms of Weyl group orbits in a maximal torus or Cartan subalgebra. Via this isomorphism, the theorem becomes Theorem \ref{thm_application_of_looijenga}. Our problem, then, is to lift this construction so that it works over any field. This also explains the need for the `reference data' described at the beginning of \S \ref{sec_construction_of_orbits}: it will provided the correct rigidification, in analogy with what happens in the proof of Theorem \ref{thm_application_of_looijenga}.

We remark that in cases $\fre_7$ and $\fre_6$, the functor $\cS$ is representable (as the triples $(C, P, t)$ have no automorphisms). This implies that for any field $k$, the map $\cS(k) \to G(k) \backslash V^\text{s}(k)$ is injective, and the composite $\cS(k) \to G(k) \backslash V^\text{s}(k) \to (V^\text{s} \! \! \sslash \! G)(k)$ is bijective.
\begin{proof}
Let us first treat the $E_7$ case. Let $(C, P) \in \cS(k)$, and let $V = \Lambda/2\Lambda$. The point $\kappa_C$ defined above lies in $T(k)$, where $T = \Hom(\Lambda, \bbG_m)$, and is well-defined up to inversion. We are going to define an extension $\widetilde{V}$ of $V$, with $\Gamma_k$-action, and then apply the constructions of \S \ref{sec_group_with_involution} to build a group around the torus $T$. Let $\widetilde{H}_\cL$ be the Heisenberg group defined in \S \ref{sec_heisenberg_and_descent}; it fits into an exact sequence
\[ \xymatrix@1{ 1\ar[r] & \bbG_m \ar[r] & \widetilde{H}_\cL \ar[r] & \Pic^0(C)[2] \ar[r] & 1.} \]
According to Proposition \ref{prop_plane_quartic_commutative_diagram}, there is a canonical injection $\Pic^0(C)[2] \hookrightarrow V^\vee$ of finite $k$-groups. Dualizing, we obtain a surjection $V \to \Pic^0(C)[2]$, and we push out the above extension by this surjection to obtain a central extension
\[ \xymatrix@1{ 1 \ar[r] & \bbG_m \ar[r] & \widetilde{E} \ar[r] & V \ar[r] & 1.} \]
The commutator pairing of $\widetilde{E}$ descends to the natural symplectic form on $V$ (since this is true for $\widetilde{H}_\cL$, by Lemma \ref{lem_weil_pairing_defined_by_commutator}, and the kernel of $V \to \Pic^0(C)[2]$ is exactly the radical of this symplectic form). Since $V$ is endowed with a $\Gamma_k$-invariant quadratic form $q : V \to \bbF_2$, we can define a character $\chi_q : \widetilde{E} \to \bbG_m$ by the formula $\widetilde{e} \mapsto (-1)^{q(e)} \widetilde{e}^2$. This makes sense since for any $\widetilde{e} \in \widetilde{E}$, we have $\widetilde{e}^2 \in \bbG_m$. Taking $\widetilde{V} = \ker \chi_q$ then gives the desired extension
\[ \xymatrix@1{ 1 \ar[r] & \{ \pm 1 \} \ar[r] & \widetilde{V} \ar[r] & V \ar[r] & 1.} \]
(This is a slight variant on the procedure leading to the extension (\ref{eqn_extension_from_theta_characteristic}).) Note that if $W = H^0(\Pic^0(C), \cL)$, then there is a natural homomorphism of $k$-groups $\widetilde{V} \to \GL(W)$. Indeed, the group $\widetilde{H}_\cL$ acts on $W$ by definition by pullback of sections; we can then pull back this action along the homomorphism $\widetilde{V} \to \widetilde{H}_\cL$. If $k = k^s$, then this is an 8-dimensional irreducible representation of the abstract group $\widetilde{V}(k^s)$, which sends $-1$ to $-\mathrm{id}_{W_{k^s}}$.

In \S \ref{sec_group_with_involution} we have associated to the triple $(\Lambda, \widetilde{V}, W)$ a simple adjoint group $H_0$ of type $E_7$, together with a stable involution $\theta$ and maximal torus $T \subset H_0$, and a representation of $\frg_0 = \frh_0^{\theta}$ on $W$. By definition, the torus $T$ is canonically isomorphic to $\Hom(\Lambda, \bbG_m)$, and $\theta$ acts on it by $t \mapsto t^{-1}$. The group $H_0$ is split; in fact, since $\frg_0$ is a form of $\frs\frl_8$ with an $8$-dimensional representation which is defined over $k$, $\frg_0$ is split. The Lie algebras $\frg_0$ and $\frh_0$ are semisimple Lie algebras of rank $7$, so this implies that $\frh_0$ must also be split.

By Proposition \ref{prop_stable_involution}, there is an isomorphism $\varphi : H \to H_0$ satisfying $\theta_0 \varphi = \varphi \theta$. This isomorphism is defined uniquely up to $H^\theta(k)$-conjugacy. The group $H^\theta$ is disconnected, with two connected components; the non-trivial component acts on the connected component $G = (H^\theta)^\circ$ by outer automorphisms. In order to pin down the isomorphism $\varphi$ up to $G(k)$-conjugacy, we observe that $\varphi^\ast(W)$ is an irreducible 8-dimensional representation of $\frg$, which is therefore isomorphic either to the fixed `standard representation' or its dual. After possibly modifying $\varphi$, we can therefore assume that $\varphi$ carries $W$ to the standard representation of $\frg$. The isomorphism $\varphi$ is then indeed determined uniquely up to $G(k)$-conjugacy.

It follows that the orbit $G(k) \cdot \varphi^{-1}(\kappa_C) \in G(k) \backslash Y(k)$ is well-defined. (Note in particular that $\kappa_C$ is defined only up to inversion, but that $\theta$ acts on $\kappa_C$ by inversion and lies in $G(k)$ (in fact in the centre of $G(k)$), so the orbit is independent of any choices.) To complete the proof in this case, we must show that $\varphi^{-1}(\kappa_C)$ is stable (equivalently, regular semisimple in $T$), and that the map we have defined is a bijection if $k = k^s$. This follows from the discussion preceding the proof of this theorem, and Theorem \ref{thm_application_of_looijenga}.

Let us now treat the $E_6$ case. The inverse image $\pi^{-1}(\ell) = e \cup f$ of the bitangent $\ell$ at $P$ in the surface $S$ determines the root lattice $\Lambda_\ell$, and we set $V = \Lambda_\ell/2\Lambda_\ell$. The natural symplectic pairing on $V$ is non-degenerate, and the quadratic form $q : V \to \bbF_2$ arising from the form on on $\Lambda_\ell$ agrees with the quadratic form on $V$ arising from the isomorphism $V \cong \Pic^0(C)[2]$ and the odd theta characteristic $\kappa$ corresponding to $\ell$, by Proposition \ref{prop_symplectic_pairings_are_identified}. We then have the Heisenberg group $\widetilde{H}_{\cL}$:
\[ \xymatrix@1{ 1 \ar[r] & \bbG_m \ar[r] & \widetilde{H}_{\cL} \ar[r] & \Pic^0(C)[2] \ar[r] & 1.} \]
Pushing out by the isomorphism $V \cong \Pic^0(C)[2]$, we obtain an extension (isomorphic to $\widetilde{H}_{\cL}$):
\[ \xymatrix@1{ 1 \ar[r] & \bbG_m \ar[r] & \widetilde{E} \ar[r] & V \ar[r] & 1.} \]
We define a character $\chi_q : \widetilde{E} \to \bbG_m$ by the formula $\widetilde{e} \mapsto (-1)^{q(e)}\widetilde{e}^2 $, and set $\widetilde{V} = \ker \chi_q$. Then $\widetilde{V}$ is an extension
\[ \xymatrix@1{ 1 \ar[r] & \{ \pm 1 \} \ar[r] & \widetilde{V} \ar[r] & V \ar[r] & 1.} \]
We define $W = H^0(\Pic^0(C), \cL)$; then $\widetilde{V}$ acts on $W$ through the homomorphism $\widetilde{V} \to \widetilde{H}_\cL$. Applying the constructions of \S \ref{sec_group_with_involution} to the triple $(\Lambda_\ell, \widetilde{V}, W)$, we obtain an adjoint group $H_0$ of type $E_6$ equipped with a stable involution $\theta_0$, together with an action of the Lie algebra $\frg_0 = \frh_0^{\theta_0}$ on $W$. Since $\frg_0$ is an inner form of $\frs\frp_8$ and has an 8-dimensional representation defined over $k$, it must be split. This implies that $H_0$ has split rank at least 4;  by the classification of forms of $E_6$ \cite[pp. 58--59]{Tit66}, we see that $H_0$ must be quasi-split, and split by a quadratic extension. This quadratic extension is the smallest extension splitting the Galois action on $\Lambda_\ell^\vee/\Lambda_\ell$. Since the geometric irreducible components $e$ and $f$ of $\pi^{-1}(\ell)$ are defined over $k$, this action is trivial, and we see that $H_0$ is also split.

Applying Proposition \ref{prop_stable_involution} once more, we see that there is an isomorphism $\varphi_e : H \to H_0$ such that $\varphi_e \theta = \theta_0 \varphi_e$. Such an isomorphism is determined up to $H^\theta(k) = G(k)$-conjugacy (as $H^\theta$ is connected in this case). Moreover, we can assume that under the isomorphism $\varphi_e$, the minuscule representation of $H_0$ with weights in $\Lambda_\ell^\vee / \Lambda_\ell$ corresponding to $e$ is identified with the `standard representation' of $H$.

The orbit $G(k) \cdot \varphi_e^{-1}(\kappa_C)$ is then well-defined: reversing the roles of $e$ and $f$ in our construction replaces $\kappa_C = \kappa(C, P, e)$ by $\kappa(C, P, f) = \kappa(C, P, e)^{-1}$, and $\theta_0$ is an outer automorphism, acting on $\Lambda_\ell^\vee/\Lambda_\ell \cong \bbZ/3\bbZ$ as multiplication by $-1$, so we can take $\varphi_f = \varphi_e \circ \theta_0$. Then we have
\[ G(k) \cdot \varphi_f^{-1}(\kappa(C, P, f)) = G(k) \cdot \varphi_f^{-1}( \theta_0(\kappa(C, P, e)) = G(k) \cdot \varphi_e^{-1}(\kappa(C, P, e)). \]
  This shows that we have constructed a well-defined map $\cS(k) \to G(k) \backslash X(k)$. The rest of the theorem in this case follows from the discussion preceding the proof of this theorem, and Theorem \ref{thm_application_of_looijenga}.

The arguments in the Lie algebra cases are very similar, with maximal tori replaced by Cartan subalgebras. We omit the details.
\end{proof}
Fix $x = (C, P, \dots) \in \cS(k)$. Let $\pi : X \to X \! \! \sslash \! G$ denote the natural quotient map, and let $X_x = \pi^{-1} \pi(x)$. Then we know that $X_x \subset X^\text{s}$ consists of a single $G$-orbit (see \S \ref{sec_vinberg_and_richardson}), but $X_x(k)$ may break up into several $G(k)$-orbits which all become conjugate over $k^s$. Let $J_x$ denote the Jacobian of $C$. We now state our second main theorem, which shows how to construct elements of the set $G(k) \backslash X_x(k)$ from the set $J_x(k)$:
\begin{theorem}\label{thm_groups_and_descent}
With notation as above, there is a canonical map
\begin{equation}\label{eqn_orbits_from_rational_points} J_x(k)/2J_x(k) \hookrightarrow G(k)\backslash X_x(k).
\end{equation}
It is functorial in $k$ in the obvious sense.
\end{theorem}
The map (\ref{eqn_orbits_from_rational_points}) will extend the map of Theorem \ref{thm_base_orbits}, in the sense that the image of the identity element of $J_x(k)/2J_x(k)$ under (\ref{eqn_orbits_from_rational_points}) equals the image of $x \in \cS(k)$ under (\ref{eqn_base_orbit}).
\begin{proof}The proof is a twist of the proof of Theorem \ref{thm_base_orbits}, using the ideas of \S \ref{sec_heisenberg_and_descent}. We treat first the $E_7$ case. Let $A \in J_x(k)$ be a rational point. Choose $B \in J_x(k^s)$ such that $[2](B) = A$. Then the field of definition of the line bundle $t_B^\ast \cL$ is equal to $k$, and we choose a bundle $\cL_B$ over $k$ which becomes isomorphic to $t_B^\ast \cL$ over $k^s$. We continue to denote $\Lambda = \Pic(S_{k^s})$, $V = \Lambda/2\Lambda$, and associate to $\cL_B$ the Heisenberg group $\widetilde{H}_{\cL_B}$, which fits into an exact sequence
\[ \xymatrix@1{ 1 \ar[r] &  \bbG_m \ar[r] &  \widetilde{H}_{\cL_B} \ar[r] & J_x[2]\ar[r] & 1.} \]
Arguing exactly as in the proof of Theorem \ref{thm_base_orbits}, we obtain an extension
\[ \xymatrix@1{ 1 \ar[r] & \{ \pm 1\} \ar[r] & \widetilde{V}_B \ar[r] & V \ar[r] & 1,} \]
together with a homomorphism $\widetilde{V}_B \to \widetilde{H}_{\cL_B}$ through which the group $\widetilde{V}_B$ acts on the space $W_B = H^0(J_x, \cL_B)$, an 8-dimensional $k$-vector space. Over $k^s$, this defines an irreducible representation of the abstract group $\widetilde{V}_B(k^s)$.

Using the constructions of \S \ref{sec_group_with_involution}, we associate to the triple $(\Lambda, \widetilde{V}_B, W_B)$ a group $H_B$ with involution $\theta_B$, maximal torus $T_B \cong \Hom(\Lambda, \bbG_m)$, and an action of the Lie algebra $\frg_B = \frh_B^{\theta_B}$ on $W_B$. Just as in the proof of Theorem \ref{thm_base_orbits}, the existence of $W_B$ implies that the groups $H_B$ and $G_B$ are split, and $T_B(k)$ has a point $\kappa_C$, well-defined up to inversion. By Proposition \ref{prop_stable_involution}, we can find an isomorphism $\varphi_B : H \to H_B$ which intertwines $\theta$ and $\theta_B$, and under which $W_B$ corresponds to the `standard representation' of $\frg \cong \frs\frl_8$. The choice of $\varphi_B$ is then unique up to the action of $G(k)$, and we associate to the point $B$ the orbit $G(k) \cdot \varphi_B^{-1}(\kappa_C) \subset Y_x(k)$.

We observe that if $A = B = 0$, the identity of $J_x(k)$, then the above construction reduces to that of Theorem \ref{thm_base_orbits}. In general, we must show that the orbit $G(k) \cdot \varphi_B^{-1}(\kappa_C) \subset P_x(k)$ depends only on the image of $A$ in $J_x(k)/2J_x(k)$ (and not on the choice of $B$), and that distinct elements of $J_x(k)/2J_x(k)$ give rise to distinct orbits. Let $\varphi_0^{-1}(\kappa_C) \in Y_x(k)$ be the point constructed in the proof of Theorem \ref{thm_base_orbits}. Since $G(k^s)$ acts transitively on $P_x(k^s)$, a well-known principle asserts that there is a canonical bijection
\begin{equation}\label{eqn_galois_cohomology_set} G(k) \backslash Y_x(k) \cong \ker \left( H^1(k, Z_G(\varphi_0^{-1}(\kappa_C))) \to H^1(k, G) \right),
\end{equation}
under which the base orbit $G(k) \cdot \varphi_0^{-1}(\kappa_C)$ corresponds to the marked element; see, for example, \cite[Proposition 1]{Bha12a}. By \cite[Corollary 2.10]{Tho13} and Proposition \ref{prop_plane_quartic_commutative_diagram}, there is a canonical isomorphism
\[ Z_G(\varphi_0^{-1}(\kappa_C)) \cong Z_{G_0}(\kappa_C) \cong \text{image }(V \to V^\vee) \cong J_x[2]. \]
We will show that under the composite
\[ G(k) \backslash Y_x(k) \hookrightarrow  H^1(k, Z_G(\varphi_0^{-1}(\kappa_C))) \cong H^1(k, J_x[2]), \]
the orbit $G(k) \cdot \varphi_B^{-1}(\kappa_C)$ is mapped to the image of $A$ under the 2-descent homomorphism of \S \ref{sec_heisenberg_and_descent}.

The pullback $t_B^\ast$ defines a canonical isomorphism $\widetilde{V} \cong \widetilde{V}_B$ over $k^s$ by the formula of  (\ref{eqn_isomorphism_of_heisenberg_groups}). This gives rise to an isomorphism of triples $F : (H_0, \theta_0, T_0) \cong (H_B, \theta_B, T_B)$ which induces the identity on $\Hom(\Lambda, \bbG_m)$ under the identification of this torus with $T_0$ and $T_B$. According to Lemma \ref{lem_functoriality_of_isomorphisms}, we can identify $F^{-1} {}^\sigma F$ with an element of $V^\vee$. Lemma \ref{lem_descent_via_theta_groups} now implies that this element in fact lies in the image of the homomorphism $V \to V^\vee$, and that under the identification of this image with $J_x[2]$, is identified with the cocycle $\sigma \mapsto {}^\sigma B - B$. This identity of cocycles implies the desired identity of cohomology classes, and completes the proof in this case.

The proof of the theorem in the remaining cases $E_6$, $\fre_7$, and $\fre_6$ simply requires analogous modifications to the proof of Theorem \ref{thm_base_orbits}. We work out the $E_6$ case here. Let us therefore take $x = (C, P) \in \cS(k)$, so that $P$ is a point such that $T_P C = \ell$ is a bitangent in the canonical embedding of the curve $C$. The root lattice $\Lambda_\ell$ is defined, and we define $V = \Lambda_\ell/2\Lambda_\ell$. The natural symplectic pairing on $V$ is non-degenerate, and the quadratic form $q : V \to \bbF_2$ arising from the form on on $\Lambda_\ell$ agrees with the quadratic form on $V$ arising from the isomorphism $V \cong J_x[2]$ and the odd theta characteristic $\kappa$ corresponding to $\ell$, by Proposition \ref{prop_symplectic_pairings_are_identified}. Let $A \in J_x(k)$, and choose a point $B \in J_x(k^s)$ with $[2](B) = A$. Let $\cL_B$ be a descent of the line bundle $t_B^\ast \cL$ to $k$. We then have the Heisenberg group $\widetilde{H}_{\cL_B}$:
\[ \xymatrix@1{ 1 \ar[r] & \bbG_m \ar[r] & \widetilde{H}_{\cL_B} \ar[r] & J_x[2] \ar[r] & 1.} \]
Arguing exactly as in the proof of Theorem \ref{thm_base_orbits}, we obtain an extension
\[ \xymatrix@1{ 1 \ar[r] & \{ \pm 1 \} \ar[r] & \widetilde{V}_B \ar[r] & V \ar[r] & 1,} \]
and $\widetilde{V}_B$ acts on the 8-dimensional $k$-vector space $W_B = H^0(J_x, \cL_B)$ through a homomorphism $\widetilde{V}_B \to \widetilde{H}_{\cL_B}$. We can apply the constructions of \S \ref{sec_group_with_involution} to the triple $(\Lambda_\ell, \widetilde{V}_B, W_B)$ to obtain a group $H_B$ with involution $\theta_B$, maximal torus $T_B \cong \Hom(\Lambda, \bbG_m)$, and an action of the Lie algebra $\frg_B = \frh_B^{\theta_B}$ on $W_B$. The existence of $W_B$ implies that the groups $H_B$ and $G_B$ are split, and $T_B(k)$ has a point $\kappa_C = \kappa(C, P, e)$ which depends on a choice of component $e$ of $\pi^{-1}(\ell) = e \cup f$. By Proposition \ref{prop_stable_involution}, we can find an isomorphism $\varphi_{B, e} : H \to H_B$ which intertwines $\theta$ and $\theta_B$, and under which the `standard representation' of $\frh$ corresponds to the minuscule representation of $\frh_B$ corresponding to the class of  $e$ in $\Lambda_\ell^\vee/\Lambda_\ell$. The choice of $\varphi_{B, e}$ is then unique up to the action of $G(k)$, and we associate to the point $B$ the orbit $G(k) \cdot \varphi_{B, e}^{-1}(\kappa(C, P, e)) \subset Y_x(k)$. Just as in the $E_7$ case, we can check that the map $B \mapsto G(k) \cdot \varphi_{B, e}^{-1}(\kappa(C, P, e))$ descends to an injection $J_x(k)/2J_x(k) \hookrightarrow G(k) \backslash Y_x(k)$. This completes the proof.
\end{proof}
\subsection{An example}\label{sec_reals}
To illustrate our theorem, we describe explicitly what happens in the $\fre_6$ case, when $k = \bbR$. Then the reference group $H$ is a split adjoint group of type $E_6$ over $\bbR$, $H^\theta = G$ is isomorphic to $\PSp_8$, a projective symplectic group in 8 variables, and $V = \frh^{d \theta = -1}$ is a 42-dimensional irreducible subrepresentation of $\wedge^4(8)$. The corresponding family of curves is the family $(C, P, t)$ of smooth non-hyperelliptic genus 3 curves, equipped with a point $P$ which is a hyperflex in the canonical embedding, and a non-zero Zariski tangent vector $t \in T_PC$. It consists of the smooth members in the family
\[ y^3 = x^4 + y(p_2x^2 + p_5x + p_8) + p_6x^2 + p_9x + p_{12} \]
(here we are using the affine chart which makes $P$ the unique point at infinity). For each tuple
\[ (p_2, p_5, p_8, p_6, p_9, p_{12}) \in \bbR^6 \]
 for which this curve is smooth, we can write down the following data:
\begin{itemize}
\item Topological invariants of the curve $C(\bbR) \subset \bbP^2(\bbR)$: following \cite{Gro81}, we write $n(C)$ for the number of connected components of $C(\bbR)$, and $a(C) = 0$ or $1$ depending on whether or not $C(\bbC) - C(\bbR)$ is disconnected.
\item A stable $G$-orbit $V_x \subset V^s$, and an $H(\bbR)$-conjugacy class of maximal tori $T \subset H$ ($T$ is the stabilizer in $H$ of the base orbit in $V_x(\bbR)$, which is regular semisimple).
\item An injection $J(\bbR) / 2 J (\bbR) \hookrightarrow G(\bbR) \backslash V_x(\bbR)$, where $J$ is the Jacobian of the curve $C$.
\end{itemize}
The isomorphism classes of tori in $H$ are in bijection with the conjugacy class of elements in the Weyl group $W$ of order 2 \cite[\S 6]{Ree11}. It turns out that these correspond to the possible topological types of the curve $C(\bbR)$ in $\bbP^2(\bbR)$, as follows:
\[ \begin{array}{c|ccccc} \text{conjugacy class} & n(C) & a(C)  & \text{no. of real bitangents} & \# J(\bbR)/2J(\bbR) & \# G(\bbR) \backslash V_x(\bbR) \\
1 & 4 & 0 & 28 & 2^3 & 36\\
s_1 & 3 & 1 & 16 & 2^2 & 10\\
s_1 s_2 & 2 & 1 & 8 & 2 & 3 \\
s_1 s_2 s_3 & 1 & 1 & 4 & 1 & 1 \\
\tau & 2 & 0 & 4 & 2 & 3
\end{array} \]
The table should be interpreted as follows: suppose that a curve $C$ has the given invariants. (It follows from the table on \cite[p. 174]{Gro81} that the only possible values for the pair $(n(C), a(C))$ are the ones listed above.) Then the real structure on the torus $T$ is the one determined by the Weyl element in the left-hand column, and the data in the remaining three columns is as given. Here $s_1, s_2, s_3 \in W$ are commuting simple reflections, and $\tau \in W$ may be constructed as follows: choose a $D_4$ root system inside $\Lambda$. Then $-1 \in W(D_4)$, and $\tau$ is the element that acts as $-1$ on the span of the $D_4$ roots, and as $+1$ on their orthogonal complement.  The elements $1, s_1, s_1 s_2, s_1 s_2 s_3$, and $\tau$ are pairwise non-conjugate in $W$ and every involution in $W$ is conjugate to one of these. (For the classification of conjugacy classes of involutions in Weyl groups, see \cite{Ric82a}.)

One can check explicitly that each of the above combinations of $(n(C), a(C))$ does indeed occur. The table can be verified as follows. It follows from our theory that there is an isomorphism $J[2](\bbC) \cong \Lambda_\ell / 2 \Lambda_\ell$ under which the action $\tau$ of complex conjugation corresponds to the action of an involution $w \in W(\Lambda_\ell) = W$ and which identifies the Weil pairing on the left-hand side with the natural symplectic pairing on the right. On the other hand, \cite[Proposition 4.4]{Gro81} shows that the data of the pair $(J[2](\bbC), \tau)$ (as symplectic $\bbF_2$-vector space with involution) is sufficient to recover $n(C)$ and $a(C)$. A calculation shows that the Weyl involutions biject with the possible choices for the pair $(n(C), a(C))$. This determines the number of real bitangents and the quantity $\# J(\bbR) / 2 J(\bbR)$.

We justify the final column using the results in the appendix. The set $G(\bbR) \backslash V_x(\bbR)$ is in canonical bijection with the set $\ker( H^1(\bbR, J[2]) \to H^1(\bbR, G) )$, the marked element corresponding to the trivial element of $J(\bbR) / 2 J (\bbR)$. We analyze this kernel using the diagram of $\bbR$-groups with exact rows, whose existence is asserted by the main result in the appendix to this paper:
\[ \xymatrix{ 1 \ar[r] & \mu_2 \ar[r] & \Sp_8 \ar[r] & \PSp_8 \ar[r] & 1 \\
1 \ar[r] & \mu_2 \ar[r]\ar[u] & \widetilde{V}\ar[r]\ar[u] & J[2] \ar[r]\ar[u] & 1, } \]
where $\widetilde{V}$ is the extension used in the proof of Theorem \ref{thm_base_orbits}; it is a subgroup of the Heisenberg group $\widetilde{H}_\cL$. Using the triviality of the set $H^1(\bbR, \Sp_8)$, we get an identification
\[ G(\bbR) \backslash V_x(\bbR) \cong \ker( H^1(\bbR, J[2]) \to H^1(\bbR, G) ) \cong \ker (H^1(\bbR, J[2] ) \to H^2(\bbR, \mu_2)), \]
where the arrow $q : H^1(\bbR, J[2] ) \to H^2(\bbR, \mu_2) \cong \bbZ / 2 \bbZ$ is the connecting map arising from the bottom row of the above commutative diagram. (Note that we are working here with non-abelian Galois cohomology; the connecting map is defined, because $\mu_2$ is central, but it need not be a homomorphism of groups.)

Tate duality gives a perfect pairing on $H^1(\bbR, J[2] )$, with respect to which $J(\bbR) / 2 J(\bbR)$ is a maximal isotropic subspace. The map $q$ is a quadratic refinement of this pairing, in the sense of \S \ref{sec_quadratic_forms}, which is identically zero on the subspace $J(\bbR) / 2 J(\bbR)$ (see \cite[Corollary 4.7]{Poo12}). It follows that $a(q) = 0$, and the set $q^{-1}(0)$ has $2^{g-1}(2^g+1)$ elements, where $g = \dim_{\bbF_2} J(\bbR) / 2 J(\bbR)$. This leads to the final column in the above table.

\newpage
\begin{appendix}
\section*{Appendix A. A converse to Lurie's functorial construction of simply laced Lie algebras}
\addcontentsline{toc}{section}{Appendix A. A converse to Lurie's functorial construction of simply laced Lie algebras. By Tasho Kaletha}
\setcounter{section}{1}
\setcounter{theorem}{0}
\subsection*{By Tasho Kaletha\footnote{This research is supported in part by NSF grant DMS-1161489.}}
In \S \ref{sec_group_with_involution} a construction due to Lurie was recalled, which associates in a functorial way a semi-simple Lie algebra $\mf{h}$ to a simply laced root lattice $\Lambda$ equipped with an extension $\tilde V$ of $V=\Lambda/2\Lambda$ by $\{\pm 1\}$. In fact, the construction produces not just $\mf{h}$, but also some additional structure, including a Cartan subalgebra $\mf{t}$. This construction was moreover refined in several ways. It was shown that an action of the Galois group of a field $k$ on $\tilde V$ is translated to a $k$-structure on $\mf{h}$; it was shown that $\mf{h}$ comes equipped with a stable involution $\theta$ (i.e.\ an involution satisfying the first condition of Proposition \ref{prop_stable_involution}); and finally a construction was described that produces from a rational representation $\rho$ of the finite algebraic $k$-group $\tilde V$ with $\rho(-1)=-1$ a rational representation $d\pi$ of the Lie-algebra $\mf{g}=\mf{h}^\theta$.

The purpose of this appendix is to provide a converse to this refinement of Lurie's construction. The basic question is: given $\mf{h}$, $\mf{t}$, and $\theta$, is it possible to recover the extension $\tilde V$ in a concrete way? That this should be the case, and in fact where the extension is to be found, was suggested to us by Jack Thorne. His idea was that the extension $\tilde V$ should be the preimage in $G_\tx{sc}$ of the 2-torsion subgroup of $T_\tx{sc}$, where $T_\tx{sc}$ is the maximal torus of the simply connected group $H_\tx{sc}$ with Lie-algabra $\mf{h}$ given by the Cartan subalgebra $\mf{t}$, and $G_\tx{sc}$ is the simply connected group with Lie-algebra $\mf{g}$. In this appendix we will show that this preimage is indeed an extension of $V$ by $\{\pm 1\}$ and we will moreover construct an isomorphism from this extension to $\tilde V$ that preserves the action of the Galois group of $k$ and intertwines the representations $\rho$ and $\pi$.

We thank Jack Thorne for sharing with us this interesting question and for including our results into his paper.

\subsection{Statement of two propositions}

Let $k$ be a field of characteristic 0, $k^s$ a fixed separable closure, $\Gamma_k=\tx{Gal}(k^s/k)$.
Let $\Lambda$ be a finite free $\Z$-module equipped with a symmetric bilinear form $\<-,-\> : \Lambda \otimes \Lambda \rw \Z$ and satisfying the conditions
\begin{itemize}
\item $\tx{rk}\Lambda>1$.
\item For any non-zero $\lambda \in \Lambda$, the value $\<\lambda,\lambda\>$ is a positive even integer.
\item The set $\Gamma = \{\lambda \in \Lambda| \<\lambda,\lambda\>=2\}$ generates $\Lambda$.
\end{itemize}
As discussed in \cite{Lur01}, these are precisely the root lattices of simply laced root systems. Here we are excluding the system $A_1$. The subset $\Gamma \subset \Lambda$ is the set of roots. We shall place the additional assumption that $\Gamma$ is irreducible. This assumption is made just for convenience and can easily be removed.

Write $q(\lambda)=\frac{1}{2}\<\lambda,\lambda\>$, this is a quadratic form. Let $V=\Lambda/2\Lambda$ and let
\[ 1 \rw \{\pm 1\} \rw \tilde V \rw V \rw 0 \]
be an extension of groups (we write the group law of $\tilde V$ multiplicatively) with the property that for each $\tilde v \in \tilde V$ and its image $v \in V$, the equality $\tilde v^2=(-1)^{q(v)}$ holds. This equation characterizes the isomorphism class of this extension.

Assume we are given an action of $\Gamma_k$ on $\Lambda$ that preserves $\<-,-\>$, as well as an action of $\Gamma_k$ on $\tilde V$ that preserves the subgroup $\{\pm 1\}$, such that the two actions on $V$ induced from these coincide. Let $\tilde\Lambda = \Lambda \times_V \tilde V$ and let $\tilde\Gamma \subset\tilde\Lambda$ be the preimage of $\Gamma$. The extension $\tilde\Lambda$ of $\Lambda$ by $\{\pm 1\}$ inherits an action of $\Gamma_k$ and this action preserves $\tilde\Gamma$.

Let $\mf{h}$ be the Lie algebra associated to this data as described in \S \ref{sec_group_with_involution}. It comes equipped with a Cartan subalgebra $\mf{t}$ and a map $\tilde\Gamma \rw \mf{h}$ sending each $\tilde\gamma$ to a non-zero root vector $X_{\tilde\gamma} \in \mf{h}_\gamma$ and having the properties
\begin{itemize}
\item $X_{-\tilde\gamma}=-X_{\tilde\gamma}$;
\item $[X_{\tilde\gamma},X_{\tilde\gamma'}]=X_{\tilde\gamma\tilde\gamma'}$ if $\gamma+\gamma' \in \Gamma$ (by assumption $\tilde\gamma\tilde\gamma' \in \tilde\Gamma$);
\item $[X_{\tilde\gamma},X_{\tilde\gamma'}]=(\tilde\gamma\tilde\gamma')H_\gamma$ if $\gamma'=-\gamma$, where $H_\gamma \in \mf{t}$ is the coroot for $\gamma$ (by assumption $\tilde\gamma\tilde\gamma' \in \{\pm 1\}$).
\end{itemize}

Let $H=\tx{Aut}(\mf{h})^\circ$ be the corresponding adjoint group, $H_\tx{sc}$ its simply connected cover, and $\theta$ the involution of $\mf{h}$ which acts by $-1$ on $\mf{t}$ and by $\theta(X_{\tilde\gamma})=-X_{\tilde\gamma^{-1}}$ on the root subspaces. It induces an involution on $H$ and $H_\tx{sc}$ as well and this involution acts by inversion of the maximal tori $T$ and $T_\tx{sc}$ whose Lie algebra is $\mf{t}$. Let $\mf{g}=\mf{h}^\theta$ be the fixed Lie subalgebra and $G=H^{\theta,\circ}$ the connected component of the fixed subgroup. Let $G'=H_\tx{sc}^\theta$. According to \cite[Theorem 8.1]{Ste68} $G'$ is connected. Its image in $H$ is equal to $G$. Since $\theta$ commutes with the action of $\Gamma_k$, the groups $G$ and $G'$ are defined over $k$.

\begin{proposition} \label{pro:fund} The group $G'$ is semi-simple and its fundamental group has order 2. \end{proposition}
Let $G_\tx{sc}$ be the simply connected cover of $G$. We will from now on denote the fundamental group of $G'$ by $\{\pm 1\} \subset G_\tx{sc}$. For a root $\gamma \in \Gamma$, let $\gamma^\vee$ be the corresponding coroot. The map
\[ V \rw T_\tx{sc},\qquad [\gamma] \mapsto \gamma^\vee(-1) \]
identifies $V$ with the 2-torsion subgroup of $T_\tx{sc}$ and this subgroup belongs to $G'$. We form the pull-back extension
\[ \xymatrix{
1\ar[r]&\{\pm 1\}\ar[r]&G_\tx{sc}\ar[r]&G'\ar[r]&1\\
1\ar[r]&\{\pm 1\}\ar@{=}[u]\ar[r]&X\ar[r]\ar[u]&V\ar[r]\ar[u]&1
} \]
This extension inherits an action of $\Gamma_k$.

Finally, given a rational representation $\rho : \tilde V \rw \tx{GL}(W)$ of the algebraic $k$-group $\tilde V$ on a finite-dimensional $k$-vector space $W$ such that $\rho(-1)=-1$, we define a representation $d\pi : \mf{g} \rw \mf{gl}(W)$ by $d\pi(X_{\tilde\gamma}-X_{\tilde\gamma^{-1}})=\rho(\tilde\gamma)/2$, and let $\pi : G_\tx{sc} \rw \tx{GL}(W)$ be the corresponding rational representation of $G_\tx{sc}$. Recall that Proposition \ref{pro:lie_alg_rep} asserts that $d\pi$ is indeed a Lie-algebra representation.

\begin{proposition} \label{pro:ext} There exists an isomorphism of extensions $\Phi : \tilde V \rw X$ which is $\Gamma_k$-equivariant and intertwines $\rho$ with $\pi|_X$ for all representations $\rho$ as above. \end{proposition}

\subsection{Proof of Proposition \ref{pro:fund}}\label{sec_proof_of_pro:fund}
According to \cite[\S 5.3]{Gro12}, the involution $\theta$ is stable and hence its conjugacy class is uniquely determined. A description of this conjugacy class for each Dynkin type is given in \cite[\S 8]{Gro12} in terms of Kac diagrams. The normalized Kac diagram of the stable involution contains a unique node with label 1, and all other nodes have label 0. According to \cite[\S 3.7]{Ree10}, this implies that the center of $G$ is finite. Thus $G$, and hence also $G'$ is semi-simple. Its Dynkin diagram is obtained by removing the unique node with label 1 from the Kac diagram of the stable involution. In order to prove that the fundamental group of $G'$ has order 2, we argue as follows. According to \cite[\S 3.7]{Ree10}, the order of the center of $G$ is given by $b_\iota$, where $\iota$ is the index of the unique node with label 1 in the Kac diagram, and $b_\iota$ is an integer defined in \cite[\S 3.3]{Ree10}, which according to Theorem 3.7 in loc. cit. is equal to $2$ if $\theta$ is inner and to $1$ if $\theta$ is outer. Since $\theta$ acts by $-1$ on the Cartan subalgebra $\mf{t}$, it is inner if and only if $-1$ belongs to the Weyl group of $(\mf{t},\mf{h})$.

The kernel of the map $G' \rw G$ is equal to $Z(H_\tx{sc})^\theta$. Thus the center of $G'$ has size $|Z(H_\tx{sc})^\theta|\cdot b_\iota$. The proof will be complete once we show that this number is equal to one half of the connection index of the Dynkin diagram of $G$. This can be done by inspection of the individual cases $A_n,n>1$, $D_n$, $E_6$, $E_7$, $E_8$. We give the examples of the exceptional types $E_6$, $E_7$, and $E_8$, and leave the discussion of the classical types $A_n$ and $D_n$ to the reader.

%For type $A$, one has to distinguish the even and odd cases. For $A_{2n}$, the Kac diagram of the stable involution is given in the second row of Table 10 of \cite{RLYG12} and equals $[1 \Rightarrow 0~0~\dots~0~0\Rightarrow0]$, with the number of $0$ equal to $n$. Removing the node labelled $1$ leaves a Dynkin diagram of type $B_n$. Since $-1$ is not in the Weyl group of $A_{2n}$, the integer $b_\iota$ is equal to $1$ and thus the center of $G$ is trivial. The automorphism $\theta$ has no fixed points on the center of $H_\tx{sc} \cong \tx{SL}_{2n+1}$, so the center of $G'$ is also trivial, i.e. $G' \cong \tx{SO}_{2n+1}$.
%
%For $A_{2n-1}$, the Kac diagram of $\theta$ is given by $D_n \Leftarrow 1$, where $D_n$ denotes a diagram of type $D_n$ whose vertices are all labelled by $0$. See third row of Table 11 with $k=n-1$. Thus $G$ is of type $D_n$. As for the case of $A_{2n}$, the center of $G$ is trivial, but now $\theta$ has fixed points of order 2 in the center of $H_\tx{sc}$, so the center of $G'$ has order 2.

For type $E_6$, the Kac diagram of $\theta$ is given by the last row of Table 3 of \cite[\S 8.1]{Gro12} and has the form $0~0~0\Leftarrow 0~1$, so $G$ has type $C_4$. Since $\theta$ is outer, $G$ is adjoint. There are no $\theta$-fixed points in the center of $H_\tx{sc}$, thus $G' \cong \tx{PSp}_4$.

For type $E_7$, the Kac diagram of $\theta$ is given by the last row of Table 4 and has the form $\tiny\arraycolsep=1pt\begin{matrix} 0&0&0&0&0&0&0\\&&&1 \end{matrix}$, so $G$ is of type $A_7$. The center of $G$ has now order 2, because $\theta$ is inner, and moreover the fixed points of $\theta$ in $Z(H_\tx{sc})$ also have order $2$, so the center of $G'$ has order $4$.

For type $E_8$, the Kac diagram of $\theta$ is given by the last row in Table 5 and has the form $\tiny\arraycolsep=1pt\begin{matrix} 1&0&0&0&0&0&0&0\\&&0 \end{matrix}$, so $G$ is of type $D_8$. The center of $G$ has order 2, because $\theta$ is inner. Since $Z(H_\tx{sc})=1$, the center of $G'$ also has order 2.

For the classical types, the relevant diagrams are those in row 2 of Table 10 ($H$ is of type $A_{2n}$ and $G$ is of type $B_n$), row 3 of Table 11 with $k=n-1$ ($H$ is of type $A_{2n-1}$ and $G$ is of type $D_n$), row 3 of Table 14 for $k=n$ even ($H$ is of type $D_n$ and $G$ is of type $D_\frac{n}{2} \times D_\frac{n}{2}$), and row 3 of Table 15 with $l=n$ odd ($H$ is of type $D_n$ and $G$ is of type $B_\frac{n-1}{2} \times B_\frac{n-1}{2}$). Note that $\theta$ is inner for $D_\tx{even}$ and outer for $A_n$ and $D_\tx{odd}$.

\subsection{Proof of Proposition \ref{pro:ext}}

\subsubsection{The group $\tx{SO}_n$}
We define the group $\tx{SO}_n$ to be the subgroup of $\tx{SL}_n$ fixed by the transpose-inverse automorphism. This group is semi-simple when $n>2$. For $n=2$, it is non-canonically isomorphic to $\mb{G}_m$ over $k^s$. One can specify an isomorphism by fixing a 4-th root of unity $i \in k^s$. Then we have
\[ \mb{G}_m \rw \tx{SO}_2,\qquad x \mapsto \frac{1}{2}\begin{bmatrix}x+x^{-1}&i(x-x^{-1})\\ -i(x-x^{-1})&x+x^{-1}\end{bmatrix}. \]
For future reference, we record the formula
\[ \begin{bmatrix} a&b\\ -b&a \end{bmatrix}^2 = \begin{bmatrix}a^2-b^2& 2ab\\ -2ab & a^2-b^2 \end{bmatrix} \]
for the squaring map $\tx{SO}_2 \stackrel{(\ )^2}{\lrw} \tx{SO}_2$.

\subsubsection{Construction of the isomorphism $\tilde V \rw X$} \label{sec:const}
Choose a set of simple roots $\Delta \subset \Gamma$. The image $\Delta_V$ of $\Delta$ in $V$ is a set of generators for this group, and the relations on this set are $2v=0$ for all $v \in \Delta_V$. Let $\tilde\Delta$ be the preimage of $\Delta$ in $\tilde\Lambda$, and $\tilde\Delta_V$ be the image of $\tilde\Delta$ in $\tilde V$. Then $\tilde\Delta_V$ is a set of generators for $\tilde V$, and the relations on this set are $\tilde v^2=(-1)$ and $\tilde v\tilde w = (-1)^{\<v,w\>}\tilde w\tilde v$.

We now define a map $\phi : \tilde\Delta \rw X$. Given $\tilde\gamma \in \tilde\Delta$ we obtain a monomorphism $\eta_{\tilde\gamma} : \tx{SL}_2 \rw H_\tx{sc}$ with $\theta$-stable image that translates the action of $\theta$ on its image to the action of transpose-inverse on $\tx{SL}_2$. The fixed subgroup $\tx{SO}_2$ of this action therefore lands in $G'$.

\begin{lemma} \label{lem:conn} The preimage of $\eta_{\tilde\gamma}(\tx{SO}_2)$ in $G_\tx{sc}$ is connected. \end{lemma}

The proof of this lemma will be given in section \ref{sec:conn}. Granting this lemma, it follows from Proposition \ref{pro:fund} that there exists a unique homomorphism $\phi_{\tilde\gamma}: \tx{SO}_2 \rw G_\tx{sc}$ making the following diagram commute.
\[ \xymatrix{
\tx{SO}_2\ar[r]^{\phi_{\tilde\gamma}}\ar[d]^{(\ )^2}&G_\tx{sc}\ar[d]\\
\tx{SO}_2\ar[r]^{\eta_{\tilde\gamma}}&G'
} \]
This homomorphism is injective.
We let $\phi(\tilde\gamma) = \phi_{\tilde\gamma}\left(\begin{bmatrix} &1\\ -1 \end{bmatrix}\right)$. By the above diagram, the image of  $\phi(\tilde\gamma)$ in $G'$ is equal to $\gamma^\vee(-1)$, which shows that $\phi(\tilde\gamma) \in X$. Moreover, $\phi(\tilde\gamma)^2=\phi_{\tilde\gamma}(-1)$ is a non-trivial element of $G_\tx{sc}$ whose image in $G'$ is trivial, hence $\phi(\tilde\gamma)^2=-1$.

We thus obtain a map $\phi : \tilde\Delta \rw X$ which descends to a map $\Phi : \tilde \Delta_V \rw X$ and whose image contains a set of generators for $X$. We claim that $\Phi$ is $\Gamma_k$-equivariant. Given $\sigma \in \Gamma_k$ we have $\eta_{\sigma\tilde\gamma}=\sigma\circ\eta_{\tilde\gamma}$, and hence $\phi_{\sigma\tilde\gamma}=\sigma\circ\phi_{\tilde\gamma}$, where on the right sides of these equations $\sigma$ denotes the action of $\sigma$ on $G'$ and $G_\tx{sc}$ respectively. Thus $\phi(\sigma\tilde\gamma)=\sigma\phi(\tilde\gamma)$ for all $\tilde\gamma \in\tilde\Gamma$ and this establishes the $\Gamma_k$-equivariance of $\Phi$.

Our task is to show that $\Phi$ respects the relation $\tilde v\tilde w = (-1)^{\<v,w\>}\tilde w\tilde v$. Once this is done, it will extend to a surjective homomorphism $\Phi : \tilde V \rw X$, which will then have to be bijective because its source and target have the same cardinality. It will furthermore be $\Gamma_k$-equivariant.

\subsubsection{The isomorphism $\tx{PGL}_2 \rw \tx{SO}_3$} \label{sec:isopgl2}
Consider the adjoint action of $\tx{PGL}_2$ on its Lie-algebra $\mf{sl}_2$. Fix a 4-th root of unity $i \in k^s$ as well as an element $\sqrt{2} \in k^s$. The basis
\[ \sqrt{2}^{-1}\begin{bmatrix} 1\\& -1 \end{bmatrix}\qquad (i\sqrt{2})^{-1}\begin{bmatrix}&1\\-1\end{bmatrix}\qquad \sqrt{2}^{-1}\begin{bmatrix}&1\\1\end{bmatrix} \]
is an orthonormal basis for the symmetric bilinear form $\tr(AB)$ and provides an isomorphism $\tx{PGL}_2 \rw \tx{SO}_3$ defined over $k^s$, which is explicitly given by
\[ \begin{bmatrix}a&b\\c&d\end{bmatrix} \mapsto (ad-bc)^{-1}\begin{bmatrix}ad+bc&i(ac+bd)&bd-ac\\ -i(ab+cd)&\frac{a^2+b^2+c^2+d^2}{2}&i\frac{a^2-b^2+c^2-d^2}{2}\\ -(ab-cd)&i\frac{c^2+d^2-a^2-b^2}{2}&\frac{a^2-b^2-c^2+d^2}{2} \end{bmatrix}. \]
Its derivative $\mf{sl}_2 \rw \mf{so}_3$ is given by
\[ \begin{bmatrix} a&b\\c&d \end{bmatrix} \mapsto \begin{bmatrix} 0&i(b+c)&b-c\\-i(b+c)&0&2ia\\c-b&-2ia&0 \end{bmatrix}. \]

\subsubsection{The relation $\tilde v\tilde w=(-1)^{\<v,w\>}\tilde w\tilde v$} \label{sec:rel}
In section \ref{sec:const} we constructed a map $\Phi : \tilde\Delta_V  \rw X$. In order to show that it extends to an isomorphism $\tilde V \rw X$, it remains to check that for $\tilde v,\tilde w \in \tilde\Delta_V$ with images $v,w \in \Delta_V$ we have
\begin{equation} \label{eq:comm} \Phi(\tilde v)\Phi(\tilde w)=(-1)^{\<v,w\>}\Phi(\tilde w)\Phi(\tilde v). \end{equation}
Let $\tilde\gamma,\tilde\delta \in \tilde\Gamma$ be preimages of $\tilde v,\tilde w$, and $\gamma,\delta \in \Delta$ be their images. We have either $\<\gamma,\delta\>=0$ or $\<\gamma,\delta\>=-1$. In the first case, the cocharacters $\eta_{\tilde\gamma}$ and $\eta_{\tilde\delta}$ commute and hence their images are contained in a common maximal torus of $G'$. The preimage in $G_\tx{sc}$ of this maximal torus is a maximal torus of $G_\tx{sc}$ and contains the images of $\phi_{\tilde\gamma}$ and $\phi_{\tilde\delta}$, and we conclude that these two cocharacters also commute. This proves \eqref{eq:comm} in the case $\<\gamma,\delta\>=0$ and we are left with the case $\<\gamma,\delta\>=-1$. Then the elements $\{X_{\tilde \gamma^{\pm 1}},X_{\tilde \delta^{\pm 1}},X_{(\tilde \gamma\tilde \delta)^{\pm1}}\}$ generate a subalgebra of $\mf{h}$ isomorphic to $\mf{sl}_3$. Even more, there is a preferred embedding $\mu_{\tilde \gamma,\tilde \delta} : \mf{sl}_3 \rw \mf{h}$ given by
\[ \begin{bmatrix} 0&1\\ &0\\&&0\end{bmatrix} \mapsto X_{\tilde \gamma}\quad \begin{bmatrix} 0\\ &0&1\\&&0\end{bmatrix} \mapsto X_{\tilde \delta}\quad \begin{bmatrix} 0&&1\\ &0\\&&0\end{bmatrix} \mapsto X_{\tilde \gamma\tilde \delta}. \]
It integrates to an embedding
$\mu_{\tilde \gamma,\tilde \delta} : \tx{SL}_3 \rw H_\tx{sc}$. The embeddings $\eta_{\tilde \gamma},\eta_{\tilde \delta} : \tx{SL}_2 \rw H_\tx{sc}$ factor through $\mu_{\tilde \gamma,\tilde \delta}$ and give embeddings
\[ \tx{SO}_2 \rw \tx{SO}_3,\qquad \begin{bmatrix} a&b\\ -b&a \end{bmatrix} \mapsto \begin{bmatrix} a&b\\ -b& a\\ &&1 \end{bmatrix} \]
and
\[ \tx{SO}_2 \rw \tx{SO}_3,\qquad \begin{bmatrix} a&b\\ -b&a \end{bmatrix} \mapsto \begin{bmatrix} 1\\ &a&b\\ &-b& a \end{bmatrix}. \]
We compose these with the isomorphism $\tx{SO}_3 \rw \tx{PGL}_2$ of section \ref{sec:isopgl2}, for which we fix the elements $i,\sqrt{2} \in k^s$ as discussed there. This gives two embeddings $\tx{SO}_2 \rw \tx{PGL}_2$.

The first one is characterized by
\[ \begin{bmatrix} a&b\\ -b&a \end{bmatrix} \mapsto \begin{bmatrix} \alpha&\beta\\ \beta&\alpha \end{bmatrix} \]
where $\alpha^2+\beta^2=a$ and $2i\alpha\beta=b$. The composition of this with the squaring map on $\tx{SO}_2$ lifts to the map
\[ \tx{SO}_2 \rw \tx{SL}_2,\qquad \begin{bmatrix} a&b\\ -b&a \end{bmatrix} \mapsto \begin{bmatrix} a&b/i\\ b/i&a \end{bmatrix}. \]
The image of $\begin{bmatrix} 0&1\\ -1&0 \end{bmatrix}$ under this map is equal to $\begin{bmatrix} 0&-i\\ -i&0 \end{bmatrix}$.

The second embedding $\tx{SO}_2 \rw \tx{PGL}_2$ is given by
\[ \begin{bmatrix} a&b\\ -b&a \end{bmatrix} \mapsto \begin{bmatrix} \sqrt{a-ib}\\ &\sqrt{a-ib}^{-1}. \end{bmatrix} \]
Note that this is well-defined with an arbitrary choice of $\sqrt{a-ib}$. Its composition with the squaring map on $\tx{SO}_2$ lifts to the map
\[ \tx{SO}_2 \rw \tx{SL}_2,\qquad \begin{bmatrix} a&b\\ -b&a \end{bmatrix} \mapsto \begin{bmatrix} (a-ib)\\ &(a-ib)^{-1} \end{bmatrix}. \]
The image of $\begin{bmatrix} 0&1\\ -1&0 \end{bmatrix}$ under this map is equal to $\begin{bmatrix} -i\\ &i \end{bmatrix}$. The claim now follows from
\[ \begin{bmatrix} 0&-i\\ -i&0 \end{bmatrix} \cdot \begin{bmatrix} -i\\ &i \end{bmatrix} = - \begin{bmatrix} -i\\ &i \end{bmatrix} \cdot \begin{bmatrix} 0&-i\\ -i&0 \end{bmatrix}. \]

\subsubsection{Intertwining property of $\Phi : \tilde V \rw X$} \label{sec:equiv}
Let $\rho : \tilde V \rw \tx{GL}(W)$ be a rational representation of the finite algebraic $k$-group $\tilde V$ on a finite-dimensional $k$-vector space $W$, having the property that $\rho(-1)=-1$. Let $\pi : G_\tx{sc} \rw \tx{GL}(W)$ be the rational representation obtained from it. We want to show that $\Phi$ intertwines $\rho$ with $\pi|_X$. It is enough to show that, for $\tilde\gamma \in \tilde\Delta$ with image $\tilde v \in \tilde V$, we have the following equality in $\tx{GL}(W)(k^s)$:
\[ \pi(\Phi(\tilde v)) = \rho(\tilde v). \]
Let $\gamma \in \Delta$ be the image of $\tilde\gamma$. Choose $\delta \in \Delta$ with $\<\gamma,\delta\>=-1$ and let $\tilde\delta \in \tilde\Delta$ be a preimage. Let $\tilde w \in \tilde V$ be the image of $\tilde\delta$. Let $Q \subset \tilde V$ be the subgroup generated by $\tilde v,\tilde w$. It is isomorphic to the quaternion group.

Let $\mu_{\tilde\gamma,\tilde\delta} : \mf{sl}_3 \rw \mf{h}$ be the embedding determined by $\tilde\gamma$ and $\tilde\delta$ as in section \ref{sec:rel}. It determines an embedding $\mu_{\tilde\gamma,\tilde\delta} : \tx{SL}_3 \rw H_\tx{sc}$.

Decompose $W = \oplus_{i=1}^n W_i$ under $\rho|_Q$ into irreducible representations over $k^s$. The condition $\rho(-1)=-1$ forces all $W_i$ to be isomorphic to the unique 2-dimensional representation of $Q$. Moreover, by construction of $d\pi$, each subspace $W_i$ of $W$ is preserved by the action of $d\pi(\mu_{\tilde\gamma,\tilde\delta}(\mf{so}_3))$, hence also by the action of $\pi(\mu_{\tilde\gamma,\tilde\delta}(\tx{SO}_3))$. We can thus focus on a single $W_i$. Choosing a suitable basis for $W_i$ over $k^s$, we obtain from $\rho|_Q$ the embedding $Q \rw \tx{SL}_2(k^s)$ given by
\[ \tilde v \mapsto \begin{bmatrix}&-i\\-i\end{bmatrix}\qquad \tilde w \mapsto \begin{bmatrix}-i\\&i\end{bmatrix}\qquad \tilde v\tilde w \mapsto \begin{bmatrix} &1\\ -1\end{bmatrix}. \]
Reviewing the construction of $d\pi$, we see that the restriction to $W_i$ of $d\pi\circ\mu_{\tilde\gamma,\tilde\delta}$ provides the isomorphism $\mf{so}_3 \rw \mf{sl}_2$ given by
\[ \begin{bmatrix} 0&1\\ -1&0\\ &&0 \end{bmatrix} \mapsto \frac{1}{2}\begin{bmatrix}&-i\\-i\end{bmatrix},\quad \begin{bmatrix} 0\\&0&1\\ &-1&0 \end{bmatrix} \mapsto \frac{1}{2}\begin{bmatrix}-i\\&i\end{bmatrix},\quad \begin{bmatrix} 0&&1\\ &0\\ -1&&0 \end{bmatrix} \mapsto \frac{1}{2}\begin{bmatrix}&1\\-1\end{bmatrix}, \]
which one easily checks to be the inverse of the isomorphism of section \ref{sec:isopgl2}. Thus, the composition of the isomorphism $\tx{SL}_2 \rw \tx{Spin}_3$ of section \ref{sec:isopgl2} with the embedding $\mu_{\tilde\gamma,\tilde\delta} : \tx{Spin}_3 \rw G_\tx{sc}$ provides a representation of $\tx{SL}_2$ on $W_i$ which in the chosen basis of $W_i$ is given by the identity map $\tx{SL}_2 \rw \tx{SL}_2$. However, the discussion of section \ref{sec:rel} shows that $\Phi(\tilde v) \in G_\tx{sc}$ is the image of the element $\begin{bmatrix}&-i\\-i\end{bmatrix}$ under the composition of the isomorphism $\tx{SL}_2 \rw \tx{Spin}_3$ of section \ref{sec:isopgl2} with the embedding $\mu_{\tilde\gamma,\tilde\delta} : \tx{Spin}_3 \rw G_\tx{sc}$. We conclude that $\rho(\tilde v)$ and $\pi(\Phi(\tilde v))$ are represented by the same matrix in $\tx{SL}_2(k^s) \subset \tx{GL}(W_i)(k^s)$.

\subsubsection{Proof of Lemma \ref{lem:conn}} \label{sec:conn}
We note first that the statement of the lemma is equivalent to the claim that the preimage of $\gamma^\vee(-1)$ in $G_\tx{sc}$ has order 4. Indeed, if the preimage of $\eta_{\tilde\gamma}(\tx{SO}_2)$ in $G_\tx{sc}$ is connected, then identifying $\tx{SO}_2$ with $\mb{G}_m$ we obtain via pull-back along $\eta_{\tilde\gamma}$ the non-split extension $1 \rw \{\pm 1\} \rw \mb{G}_m \rw \mb{G}_m \rw 1$, and the element $\gamma^\vee(-1)$ corresponds to the element $-1$ of the right copy of $\mb{G}_m$, which evidently has two preimages of order $4$. On the other hand, if the preimage of $\eta_{\tilde\gamma}(\tx{SO}_2)$ in $G_\tx{sc}$ is disconnected, the the corresponding extension is the split extension $1 \rw \{\pm 1\} \rw \{\pm 1\} \times \mb{G}_m \rw \mb{G}_m \rw 1$ and the element $-1 \in \mb{G}_m$ has two lifts of order $2$.

We have the element $\tilde\gamma \in \tilde\Gamma$ and the corresponding element $\gamma \in \Gamma$. The chosen base $\Delta$ of $\Gamma$ in the discussion of section \ref{sec:const} will be unimportant. We first claim that there exists a maximal torus $S_\tx{sc} \subset H_\tx{sc}$, a Borel subgroup $C$ containing $S_\tx{sc}$, and a root $\alpha$ of $H_\tx{sc}$ with respect to $S_\tx{sc}$ such that $\theta$ preserves the pair $(S_\tx{sc},C)$ as well as the root $\alpha$ and $\gamma^\vee(-1)=\alpha^\vee(-1)$. Indeed, choose a base $\Delta$ for $\Gamma$ such that the corresponding Kostant cascade $M$ (see \cite{Kos11}) contains $\gamma$. For each $\beta \in M$, choose a preimage $\tilde\beta \in \tilde\Gamma$. Let
\[ g = \prod_{\beta \in M}\eta_{\tilde\beta}\begin{bmatrix}i/2&1\\-1/2&-i\end{bmatrix} \in H_\tx{sc}.\]
Then one checks that $S_\tx{sc} := \tx{Ad}(g)T_\tx{sc}$ is normalized by $\theta$. If we transport the action of $\theta$ on $S_\tx{sc}$ back to $T_\tx{sc}$ via the isomorphism $\tx{Ad}(g)$, we obtain the automorphism $\tx{Ad}(g^{-1}\theta(g))\circ\theta$ and one computes that $\tx{Ad}(g^{-1}\theta(g))$ acts as the product of reflections $\prod_{\beta\in M}s_\beta$, which according to \cite[Prop. 1.10]{Kos11} represents the longest element of the Weyl group with respect to the basis $\Delta$. This shows that $\tx{Ad}(g^{-1}\theta(g))\circ\theta$ preserves the basis $\Delta$. It also evidently fixes the root $\gamma$. Let $\alpha=\tx{Ad}(g)\gamma$, and let $C$ be the Borel subgroup corresponding to the basis $\tx{Ad}(g)\Delta$. Finally, $\alpha^\vee(-1)=\gamma^\vee(-1)$ follows from the fact that the element $g \in H_\tx{sc}$ centralizes $\gamma^\vee(-1) \in H_\tx{sc}$. Indeed, the image of $\eta_{\tilde\beta}$ for $\beta \in M \setminus \{\gamma\}$ centralizes the image of $\gamma^\vee$, while the image of $\eta_{\tilde\gamma}$ centralizes the element $\gamma^\vee(-1)$. The claim is proved.

We are now interested in showing that the preimage of $\alpha^\vee(-1)$ in $G_\tx{sc}$ has order $4$. For this it is convenient to use again the equivalent formulation that the preimage of $\alpha^\vee(\mb{G}_m)$ in $G_\tx{sc}$ is connected. By passing from $\gamma$ to $\alpha$ we are now in the more advantageos situation that this preimage belongs to the preimage in $G_\tx{sc}$ of $G' \cap S_\tx{sc} = S_\tx{sc}^\theta$, which is a maximal torus. Call this maximal torus $\tilde S \subset G_\tx{sc}$. We form the pull-back diagram
\[ \xymatrix{
1\ar[r]&\{\pm 1\}\ar[r]&\tilde S\ar[r]&S_\tx{sc}^\theta\ar[r]&1\\
1\ar[r]&\{\pm 1\}\ar@{=}[u]\ar[r]&?\ar[r]\ar[u]&\mb{G}_m\ar[u]^{\alpha^\vee}\ar[r]&1
} \]
and would like to show that the bottom extension is not split. Passing to character modules we obtain the push-out diagram
\[ \xymatrix{
0\ar[r]&X^*(S_\tx{sc})_\theta\ar[r]\ar[d]^{\alpha^\vee}&X^*(\tilde S)\ar[d]\ar[r]&\Z/2\Z\ar[r]\ar@{=}[d]&0\\
0\ar[r]&\Z\ar[r]&X^*(?)\ar[r]&\Z/2\Z\ar[r]&0
} \]
and would still like to show that the bottom extension is not split. This is equivalent to showing that for one, hence any, lift $\dot 1 \in X^*(?)$ of $1 \in \Z/2\Z$, we have $2\dot 1 \in \Z \sm 2\Z$. This in turn is equivalent to showing that for one, hence any, lift $\dot 1 \in X^*(\tilde S)$ of $1 \in \Z/2\Z$, we have $\alpha^\vee(2\dot 1) \notin \alpha^\vee(2X^*(S_\tx{sc})_\theta)$. Now $X^*(S_\tx{sc})$ is the weight lattice of the group $H_\tx{sc}$ with respect to the torus $S_\tx{sc}$. Since $\alpha^\vee$ is a coroot, we have $\alpha^\vee(X^*(S_\tx{sc})_\theta)=\Z$. Our task is then to show that the image in $\Q$ of $X^*(\tilde S)$ under $\alpha^\vee$ is not contained in $\Z$. But $X^*(\tilde S)$ is equal to the weight lattice of the group $G_\tx{sc}$ relative to the maximal torus $\tilde S$. We thus have to show that $\alpha^\vee \in X_*(S_\tx{sc})^\theta$ does not belong to the coroot lattice of $G'$.

To that end, we need to describe the root and coroot systems of $G'$. Let $R \subset X^*(S_\tx{sc})$ and $R^\vee \subset X_*(S_\tx{sc})$ be the root and coroot systems of $H_\tx{sc}$, and let $\Delta \subset R$ be the base given by the Borel subgroup $C$. We choose a non-zero root vector $X_\beta \in \mf{h}_\beta$ for each $\beta \in \Delta$ subject to the condition $X_{\theta\beta} = \theta X_\beta$ provided $\theta\beta \neq \beta$. For $\beta \in \Delta$ satisfying $\theta\beta = \beta$ we have $\theta X_\beta = \epsilon X_\beta$ with $\epsilon \in \{1,-1\}$. Letting $\{\check\omega_\beta|\beta \in \Delta\}$ be the system of fundamental coweights, we set $s \in S$ to be the product of $\check\omega_\beta(-1)$ for all $\beta \in \Delta$ with $\theta\beta = \beta$ and $\theta X_\beta = -X_\beta$. Then $s \in S^\theta$ is of order $2$ and $\theta = \tx{Ad}(s)\theta_0$, with $\theta_0$ an automorphism of $H_\tx{sc}$ preserving the splitting $(S_\tx{sc},C,\{X_\beta\})$. The root system of $G'$ is a subset $R' \subset X^*(S_\tx{sc}^\theta) = X^*(S_\tx{sc})_\theta$. The duality between $X^*(S_\tx{sc})$ and $X_*(S_\tx{sc})$ induces a duality between $X^*(S_\tx{sc})_\theta$ and $X_*(S_\tx{sc})^\theta$. The coroot system of $G'$ is a subset $R'^\vee \subset X_*(S_\tx{sc})^\theta$. The system $R' \subset X^*(S_\tx{sc})_\theta$ and its dual system $R'^\vee \subset X_*(S_\tx{sc})^\theta$ can be described using the results of \cite{Ste68}, which are summarized in \cite[\S1.1,\S1.3]{KS99}. As evident from the discussion there, the root system $A_{2n}$ behaves differently from all other root systems, a phenomenon that manifests itself in the occurrence of restricted roots of type $R_2$ and $R_3$. It is therefore convenient to treat the special case of $A_{2n}$ separately. Fortunately, this special case is rather easy.

Assuming that $R$ is of type $A_{2n}$, we enumerate $\Delta=\{\alpha_1,\dots,\alpha_{2n}\}$ with $\theta(\alpha_i)=\alpha_{2n+1-i}$. Since $\theta$ has no fixed points in $\Delta$, we have $\theta_0=\theta$. Thus the projection of $\Delta$ to $X^*(S_\tx{sc})_\theta$ forms a set of simple roots for $R'$. Let $\alpha_i' \in R'$ denote the projection of $\alpha_i$. Then $\alpha_1',\dots,\alpha_{n-1}'$ are of type $R_1$, and the corresponding coroots are given by $\alpha_i'^\vee = \alpha_i^\vee+\alpha_{2n+1-i}^\vee$. On the other hand $\alpha_n'$ is of type $R_2$ and its coroot is given by $2(\alpha_n^\vee+\alpha_{n+1}^\vee)$. It follows that the coroot lattice of $G'$ is the sublattice of $X_*(S_\tx{sc})^\theta$ spanned by the points $\{\alpha_1^\vee+\alpha_{2n}^\vee,\dots,\alpha_{n-1}^\vee+\alpha_{n+2}^\vee,2(\alpha_n^\vee+\alpha_{n+1}^\vee)\}$. On the other hand, we may assume without loss of generality that $\alpha$ is the highest root of $R$ (by making the same assumption on the root $\gamma$, bearing in mind that the highest root is always part of the Kostant cascade). Then $\alpha^\vee = \alpha_1^\vee+\dots+\alpha_{2n}^\vee$ evidently does not belong to the coroot lattice of $G'$. This completes the discussion of the case $A_{2n}$.

The remaining root systems can now be treated uniformly, because all occurring restricted roots are of type $R_1$. According to the discussion in \cite[\S1.3]{KS99}, the root system $R'$ is given by the image of the set
\[ \dot R' = \{\beta \in R| \theta\beta = \beta \Rightarrow \beta(s)=1 \} \]
under the natural projection $X^*(S_\tx{sc}) \rw X^*(S_\tx{sc})_\theta$. For the description of $R'^\vee$, we have the following lemma.
\begin{lemma} For any element of $\beta' \in R'$ represented by $\beta \in \dot R'$, the coroot $\beta'^\vee \in X_*(S_\tx{sc})^\theta$ is given by
\[ \begin{cases}
\beta^\vee&,\theta\beta = \beta \\
\beta^\vee+\theta\beta^\vee&,\theta\beta\neq\beta
\end{cases} \]
\end{lemma}
\begin{proof}
Since $\beta'$ is of type type $R_1$, we know that if $\theta\beta \neq \beta$ then $\theta\beta \perp\beta$. According to \cite[Chap. VI, \S1, no. 1]{Bou02}, $\beta'^\vee$ is the unique element of the dual space of $X^*(S_\tx{sc})_\theta \otimes \Q$ with the properties $\<\beta'^\vee,\beta'\>=2$ and $s_{\beta',\beta'^\vee}(R') \subset R'$, where $s_{\beta',\beta'^\vee}(x)=x-\<\beta'^\vee,x\>\beta'$ is the reflection determined by $\beta',\beta'^\vee$. We need to check that the elements given in the statement of the lemma satisfy these properties. The first property is immediate. For the second property we take $\beta_1,\beta_2 \in \dot R'$ and let $\beta_1',\beta_2' \in R'$ be their images. Let $\beta_1'^\vee \in X_*(S_\tx{sc})^\theta$ be given by the table above. We need to show that $s_{\beta_1',\beta_1'^\vee}(\beta_2') \in R'$. If $\beta_2$ is perpendicular to both $\beta_1$ and $\theta\beta_1$, or if $\beta_1' = \pm\beta_2'$, then the claim is clear. We thus assume that this is not the case.

If $\beta_1$ is fixed by $\theta$, then $s_{\beta_1',\beta_1'^\vee}(\beta_2')$ is the image of $s_{\beta_1,\beta_1^\vee}(\beta_2)$. This element of $R$ belongs to $\dot R'$, because it is fixed by $\theta$ precisely when $\beta_2$ is, and in this case it kills $s$, since both $\beta_1$ and $\beta_2$ do.

If $\beta_1$ is not fixed by $\theta$, but $\beta_2$ is, then we have $\<\beta_1^\vee+\theta\beta_1^\vee,\beta_2\>=2\<\beta_1^\vee,\beta_2\> =2\epsilon \neq 0$ and conclude that $s_{\beta_1',\beta_1'^\vee}(\beta_2')$ is the image of $\beta_2-2\epsilon\beta_1$, which coincides with the image of $\beta_2-\epsilon\beta_1 - \epsilon\theta\beta_1$. The latter element belongs to $R$, because $\beta_1 \perp \theta\beta_1$. It is furthermore fixed by $\theta$ and kills $s$, so belongs to $\dot R'$.

Now assume that both $\beta_1,\beta_2$ are not fixed by $\theta$. If $\<\beta_1^\vee,\beta_2\>$ and $\<\theta\beta_1^\vee,\beta_2\>$ are both non-zero and have opposite signs, then $s_{\beta_1',\beta_1'^\vee}(\beta_2')=\beta_2'$. If $\<\beta_1^\vee,\beta_2\>$ and $\<\theta\beta_1^\vee,\beta_2\>$ are both non-zero and have the same sign $\epsilon \in \{1,-1\}$, then $s_{\beta_1',\beta_1'^\vee}(\beta_2')$ is equal to the image of $\beta_2-2\epsilon\beta_1$, which coincides with the image of $\beta_2 -\epsilon\beta_1 -\epsilon\theta\beta_1$. As above this element belongs to $R$. It is moreover not $\theta$-fixed, thus belongs to $\dot R'$. It remains to consider the cases where exactly one of $\<\beta_1^\vee,\beta_2\>$ and $\<\theta\beta_1^\vee,\beta_2\>$ is non-zero. We will give the computation only in the case $\<\beta_1^\vee,\beta_2\>=0$, $\<\theta\beta_1^\vee,\beta_2\>=-1$, the other cases being analogous. The element $s_{\beta_1',\beta_1'^\vee}(\beta_2') \in X^*(S_\tx{sc})_\theta$ is equal to the image of $\beta_2+\beta_1 \in R$ and we claim that this element is not $\theta$-fixed. If it were, we'd have $\beta_2=\theta\beta_2+\theta\beta_1-\beta_1$ and applying $\<\theta\beta_1^\vee,-\>$ we would obtain $-1=0+2-0$.
\end{proof}

Armed with this lemma we complete the proof of Lemma \ref{lem:conn} as follows. We have the element $\alpha^\vee \in X_*(S_\tx{sc})^\theta$, which is a coroot for the group $H_\tx{sc}$. We wish to show that it does not belong to the coroot lattice for the group $G'$. Assume the contrary. Then inside of the lattice $X_*(S_\tx{sc})^\theta$ we have the equation $\alpha^\vee = \sum n_i\beta_i'^\vee$ for some integers $n_i$ and some roots $\beta_i' \in R'$. We choose for each $\beta_i'$ a lift $\beta_i \in \dot R'$ and apply the previous lemma, thereby obtaining
\[ \alpha^\vee = \sum n_i\beta_i^\vee+\sum n_i(\beta_i^\vee+\theta\beta_i^\vee), \]
where we have subdivided the set of $\{\beta_i\}$ into the cases corresponding to the statement of above lemma. This equation holds inside the coroot lattice of $H_\tx{sc}$. Since $R$ is a simply laced root system, the bijection $R \rw R^\vee,\beta \mapsto \beta^\vee$ extends to a $\Z$-linear bijection from the root lattice to the coroot lattice. This tells us that we have the equation
\[ \alpha = \sum n_i\beta_i+\sum n_i(\beta_i+\theta\beta_i) \]
in the root lattice of $H_\tx{sc}$, i.e. in $X^*(S)$. However, the right hand side is a character of $S$ which kills the element $s \in S$. This would imply that $\alpha \in \dot R'$, which would then imply that $\theta$ acts trivially on the root space $\mf{h}_\alpha$. This is however false, because for $X=\tx{Ad}(g)X_{\tilde\gamma} \in \mf{h}_\alpha$ we have
\[ \theta(X) =\tx{Ad}(g)\tx{Ad}(g^{-1}\theta(g))\theta(X_{\tilde\gamma})=\tx{Ad}(g)\tx{Ad}\eta_{\tilde\gamma}\begin{bmatrix}&-i\\-i\end{bmatrix}(-X_{\tilde\gamma^{-1}}) = -X. \]
The proof of Lemma \ref{lem:conn} is now complete.
\end{appendix}

\end{document}